\documentclass[12pt]{amsart}
\usepackage{amsmath,times,epsfig,amssymb,amsbsy,amscd,amsfonts,amstext,color,bm,mathrsfs}
\usepackage[margin=1in]{geometry}
\usepackage{enumerate}
\usepackage{placeins}
\usepackage{array}
\usepackage{tabulary}
\usepackage{multirow}
\usepackage{graphicx}
\usepackage{hyperref}
\usepackage{hhline}
\usepackage{soul}
\usepackage{multirow}
\usepackage[table, svgnames, dvipsnames]{xcolor}
\usepackage{makecell, cellspace, caption}
\usepackage{subcaption}
\usepackage[color,matrix,arrow]{xypic}
\usepackage{tikz}
        \usetikzlibrary{arrows.meta}
        \usetikzlibrary{arrows}
        \usetikzlibrary{quotes}
        \usetikzlibrary{graphs}
        \usetikzlibrary{positioning}
\usepackage{tikz-cd}
\hypersetup{
  colorlinks   = true, 
  urlcolor     = blue, 
  linkcolor    = blue, 
  citecolor   = red 
}

\tikzstyle{v} = [circle, draw, inner sep=1.7pt, minimum size=3pt, fill=black]

\theoremstyle{plain}
\newtheorem{theorem}{Theorem}[section]
\newtheorem{corollary}[theorem]{Corollary}
\newtheorem{lemma}[theorem]{Lemma}

\newtheorem{proposition}[theorem]{Proposition}

\makeatletter
    
    \@addtoreset{equation}{section}
  \makeatother

\theoremstyle{definition}
\newtheorem{definition}[theorem]{Definition}

\newtheorem{example}[theorem]{Example}

\theoremstyle{remark}
\newtheorem{remark}[theorem]{Remark}

\DeclareMathOperator{\Der}{Der}
\newcommand{\Ac}{\mathscr{A}}
\newcommand{\Bc}{\mathscr{B}}
\newcommand{\Dd}{\mathscr{D}}
\newcommand{\Pp}{\mathscr{P}}

\newcommand{\Hc}{\mathscr{H}}
\newcommand{\A}{\mathcal{A}}

\newcommand{\scD}{\mathcal{D}}

\newcommand{\R}{\mathbb{R}}
\newcommand{\Qc}{\mathcal{Q}}
\newcommand{\scR}{\mathcal{R}}

\newcommand{\Pc}{\mathcal{P}}

\newcommand{\Z}{\mathbb{Z}}

\newcommand{\E}{{\mathcal{E}}}

\newcommand{\vn}{\noindent}

\newcommand{\I}{{\mathcal{I}}}
\newcommand{\J}{{\mathcal{J}}}

\newcommand{\tbf}{\textbf}

\newcommand{\Hom}{\operatorname{Hom}}
\newcommand{\quasi}{\operatorname{quasi}}

\newcommand{\rk}{\operatorname{rk}}
\newcommand{\Mat}{\operatorname{Mat}}

\newcommand{\codim}{\operatorname{codim}} 
\renewcommand{\part}{\operatorname{par}}

\newcommand{\cov}{<\mathrel{\mkern-5mu}\mathrel{\cdot}}

\newcolumntype{K}[1]{>{\centering\arraybackslash}p{#1}}

\makeatletter
\@namedef{subjclassname@2020}{%
  \textup{2020} Mathematics Subject Classification}
\makeatother

\begin{document}

\title[Inductive and divisional posets]{Inductive and divisional posets}

\begin{abstract}
\label{sec:intro}
We call a poset factorable if its characteristic polynomial has all positive integer roots. Inspired by inductive and divisional freeness of a central hyperplane arrangement, we introduce and study the notion of inductive posets and their superclass of divisional posets. It then motivates us to define the so-called inductive and divisional abelian (Lie group) arrangements, whose posets of layers serve as the main examples of our posets. Our first main result is that every divisional poset is factorable. Our second main result shows that the class of inductive posets contains strictly supersolvable posets, the notion recently introduced due to Bibby and Delucchi (2022). This result can be regarded as an extension of a classical result due to Jambu and Terao (1984), which asserts that every supersolvable hyperplane arrangement is inductively free. Our third main result is an application to toric arrangements, which states that the toric arrangement defined by an arbitrary ideal of a root system of type $A$, $B$ or $C$ with respect to the root lattice is inductive. 
 \end{abstract}

\author{Roberto Pagaria}
\address{Roberto Pagaria, Universit\`a di Bologna, Dipartimento di Matematica, Piazza di Porta San Donato 5 -- 40126 Bologna, Italy}
\email{roberto.pagaria@unibo.it}

\author{Maddalena Pismataro}
\address{Maddalena Pismataro, Universit\`a di Bologna, Dipartimento di Matematica, Piazza di Porta San Donato 5 -- 40126 Bologna, Italy}
\email{maddalena.pismataro2@unibo.it}

\author{Tan Nhat Tran}
\address{Tan Nhat Tran, Institut f\"ur Algebra, Zahlentheorie und Diskrete Mathematik, Fakult\"at f\"ur Mathematik und Physik, Leibniz Universit\"at Hannover, Welfengarten 1, D-30167 Hannover, Germany.}
\email{tan.tran@math.uni-hannover.de}

\author{Lorenzo Vecchi}
\address{Lorenzo Vecchi, Universit\`a di Bologna, Dipartimento di Matematica, Piazza di Porta San Donato 5 -- 40126 Bologna, Italy}
\email{lorenzo.vecchi6@unibo.it}


\date{\today}
\maketitle

\tableofcontents

\section{Introduction}
A \emph{hyperplane arrangement} $\Hc$  is a finite set of hyperplanes ($1$-codimensional affine subspaces) in a finite dimensional vector space $V$. 
The \emph{intersection poset}  $L(\Hc)$ of $\Hc$ is the set of all nonempty intersections of hyperplanes in $\Hc$, which is often referred to as the  \emph{combinatorics} of $\Hc$. 
The arrangement $\Hc$ is called \emph{factorable} if  its \emph{characteristic polynomial} $\chi_{\Hc}(t)$ has all nonnegative integer roots.
In this case, we call the roots of $\chi_{\Hc}(t)$ the \emph{(combinatorial) exponents} of $\Hc$.

An arrangement is called \emph{central} if every hyperplane in it goes through the origin. 
A central arrangement $\Hc$ is said to be \emph{free} if its \emph{module $D(\Hc)$ of logarithmic derivations} is a free module (Definition \ref{def:free-arr}).  
A remarkable theorem connecting algebra and combinatorics of arrangements due to Terao  asserts that if an arrangement $\Hc$ is free, then it is factorable and its combinatorial exponents coincide with the \emph{degrees} of the derivations in any basis for $D(\Hc)$ (Theorem \ref{thm:Factorization}). 

\begin{definition} 
\label{def:comb-prop}
A property $P$ of arrangements is called a \emph{combinatorial property (or combinatorially determined)} if for any distinct arrangements $\Hc_1$ and $\Hc_2$ in $V$ having the same combinatorics, i.e., their intersection posets are  isomorphic $L(\Hc_1) \simeq L(\Hc_2)$, then $\Hc_1$ has property $P$ if and only if $\Hc_2$ has property $P$.
\end{definition}

Based on the factorization theorem mentioned above, Terao conjectured that freeness is a combinatorial property \cite[Conjecture 4.138]{OT92}. 
Terao's conjecture remains open till now even in dimension $3$.

A natural approach to the conjecture is to find a significant class of arrangements whose freeness is combinatorially determined. 
Motivated by the addition-deletion theorem for free arrangements \cite[Theorem 4.51]{OT92}, 
Terao first defined the class of \emph{inductively free arrangements} in which an arrangement can be built from the \emph{empty arrangement} by adding a hyperplane one at a time subject to the inductive freeness of both deleted and restricted arrangements, and a divisibility condition on the characteristic polynomials (Definition \ref{def:IF-arr}). 
A notable feature of this class due to Jambu and Terao \cite{JT84} is that it contains \emph{supersolvable arrangements} (Definition \ref{def:SS-arr}), a prominent class of arrangements defined earlier by Stanley \cite{St72}. 
Later on, Abe \cite{Abe16} proved a refinement of the addition-deletion theorem, and introduced a proper superclass of inductively free arrangements, the so-called \emph{divisionally free arrangements} (Definition \ref{def:DF-arr}). 
Both inductively and divisionally free arrangements are combinatorially determined, proper subclasses of free arrangements (Remark \ref{rem:all-property}). 
In particular, inductive or divisional freeness is a sufficient condition for the arrangement' factorability.

In recent years, there has been increasing attention towards extending the known properties of hyperplane arrangements to \emph{toric arrangements}, or more generally, to \emph{abelian arrangements}. 
Given an abelian Lie group $G=(\mathbb{S}^1)^a\times \R^b$ ($a,b \ge 0$) and a finite set $\A$ of integral vectors in $\Gamma=\Z^\ell$, Liu, Yoshinaga and the third author \cite{LTY21} defined the \emph{abelian arrangement} $\Ac =\Ac(\A, G)$ by means of group homomorphisms from $\Gamma$ to $G$ (see Section \ref{sec:ind-abel-arr} for details). 
In particular, when $G =\R$ (or  $\mathbb{C}$) we obtain a real (or complex) hyperplane arrangement, and when $G =  \mathbb{S}^1$ (or $ \mathbb{C}^\times $) this is known as a real (or complex) \emph{toric arrangement} which describes a finite set of (translated) hypertori in a finite dimensional torus. 

We recall some important results of abelian arrangements. 
In  \cite{LTY21}, a formula for the \emph{Poincar\'e polynomial} of the complement of $\Ac$ when $G$ is noncompact (i.e., $b > 0$) is given; 
this generalizes the formulas of Orlik and Solomon \cite{OS80}, and De Concini, Procesi, and Moci \cite{DP05,L12} for complex hyperplane and toric arrangements. 
(The cohomology ring structure is also known \cite{OS80,DP05,CDDMP20} in the case of hyperplane or toric arrangements.)
In  \cite{TY19}, the intersection poset (or  poset of \emph{layers}) $L(\Ac)$ of $\Ac$ is defined as the set of all connected components of intersections of elements in $\Ac$, and its characteristic polynomial is computed. 

It is well-known that the intersection poset of a central hyperplane arrangement is a \emph{geometric lattice} (Definition \ref{def:GL}). 
Bibby and Delucchi \cite{BD22} recently introduced a more general notion of \emph{(locally) geometric posets} (Definitions \ref{def:LGP} and \ref{def:GP}) and showed that these posets describe the intersection data of abelian arrangements (Theorem \ref{thm:GP}). 
Furthermore, based on an extension of the concept of lattice modularity, the authors defined the notion of \emph{strictly supersolvable posets} (Definition \ref{def:SSS}), which is of our particular interest here. 
It is proved that every strictly supersolvable poset is factorable (Theorem \ref{thm:SSS-factor}), which extends the result by Stanley for supersolvable lattices \cite{St72}. 

The first motivation for this work is a pursuit of a theory for ``free abelian arrangements". 
As of this writing, we do not know how to pass from algebraic consideration of freeness of hyperplane arrangements to abelian or just toric arrangements. 
However, at the purely combinatorial level using only information from the posets, it is possible to define and study the combinatorial structures of abelian arrangements and geometric posets in the same way that inductive freeness and divisional freeness do for hyperplane arrangements and geometric lattices.

In this paper, we give definitions of \emph{inductive} and \emph{divisional posets} as subclasses of locally geometric posets (Definitions \ref{def:IP} and \ref{def:DP}).
The former is a proper subclass of the latter owing to a deletion-restriction formula for characteristic polynomials  (Theorem \ref{thm:LGP-DC} and Proposition \ref{prop:IP<DP}). 
On the arrangement theoretic side, we define \emph{inductive} and \emph{divisional arrangements} in a similar way (Definitions \ref{def:IA} and \ref{def:DA}). 
We show that an abelian arrangement is inductive (resp., divisional) if and only if its intersection poset is inductive (resp., divisional) (Theorem \ref{thm:IP=IA-DP=DA}).
As a consequence, inductiveness and divisionality are combinatorial properties of abelian arrangements (Corollary \ref{cor:IA-DA-combinatorial}). 
 
The second motivation is a contribution to factorability of an abelian arrangement, or more generally, of a locally geometric poset (Definition \ref{def:poset-exp}). 
 Beyond ranked lattices, there are some reasons for an arbitrary poset to be factorable (e.g., \cite{Hal17}). 
Our first main result in the paper is that a divisional (in particular, an inductive) poset has this factorability. 
  \begin{theorem}
\label{thm:DP-factor-intro}
If a poset is divisional, then it is factorable.
\end{theorem}

Our second main result is a generalization of the classical result of Jambu and Terao \cite{JT84} mentioned earlier for supersolvable and inductively free arrangements. 
   \begin{theorem}
\label{thm:SSS<IP-intro}
If a poset is strictly supersolvable, then it is inductive.
\end{theorem}

Using the notion of \emph{characteristic quasi-polynomial} from \cite{KTT08}, the third author  \cite{T19} showed that the toric arrangement defined by an arbitrary \emph{ideal} of a root system of type $A$, $B$ or $C$ with respect to the \emph{root lattice} is factorable. 
Our third main result is a strengthening of this result. 

   \begin{theorem}
\label{thm:ABC-IA-intro}
The toric arrangement defined by an arbitrary ideal of a root system of type $A$, $B$ or $C$ with respect to the root lattice is inductive. 
\end{theorem}

Finally, we give a discussion on the \emph{localization} at a layer of an abelian arrangement (Section \ref{sec:local}). 
It is shown that inductive freeness of a hyperplane arrangement is preserved under taking localization \cite{HRS17}. 
We show that it is not the case for an arbitrary abelian arrangement by providing an example of an inductive toric arrangement with a non-inductive localization. 
Furthermore, this example indicates a rather interesting phenomenon that changing the base group $G$ would turn a non-inductive arrangement into an inductive one -- there exists a finite set $\A$ of integral vectors whose corresponding  hyperplane arrangement $\Ac(\R)$ is not inductive but the toric arrangement  $ \Ac(\mathbb{S}^1)$ is. 

\section{Preliminaries}

\subsection{Posets}
\label{subsec:posets}

We begin by recalling the definitions and basic facts of (locally) geometric posets and (strictly) supersolvable posets following \cite{BD22}.

All posets $(\Pc, \le_\Pc)$ will be finite and have a unique minimal element $\hat 0$. 
All $\Pc$ will also be \emph{ranked} meaning that for every $x \in \Pc$, all maximal chains among those with $x$ as greatest element have the same length, denoted  $\mathrm{rk}(x)$. 
Define the \emph{rank} of a poset $\Pc$ to be
$$\rk(\Pc):=\max \{\rk(x) \mid x \in \Pc\}.$$

The \emph{M\"{o}bius function} $\mu:=\mu_\Pc$ of a poset $\Pc$ is the map $\mu_\Pc:\Pc \times \Pc \longrightarrow\Z$ defined by
$$\mu_\Pc(a,b):= 
 \begin{cases}
1 & \mbox{ if\,  $a= b$}, \\
-\sum_{a \le c <b}\mu_\Pc(a,c) & \mbox{ if\,  $a< b$}, \\
0 & \mbox{ otherwise}.
\end{cases}
$$

The \emph{characteristic polynomial} $\chi_\Pc(t)\in \mathbb{Z}[t]$ of $\Pc$ is defined as
\begin{equation*}
\chi_{\Pc}(t):= \sum_{x \in \Pc}\mu(\hat0, x)t^{\rk(\Pc)-\rk(x)}.
\end{equation*}

\begin{definition} 
\label{def:poset-exp}
A poset $\Pc$ is \emph{factorable} if the roots of its characteristic polynomial $\chi_{\Pc}(t)$ form a subset of positive integer roots.
In this case, we call the roots of $\chi_{\Pc}(t)$ the \emph{(combinatorial) exponents} of $\Pc$ and write
 $$\exp(\Pc) = \{d_1,\ldots,d_{\rk(\Pc)}\}$$
 for the multiset of exponents.
Denote by $\mathbf{FR}$ the class of factorable posets.
\end{definition} 

The trivial lattice $\{\hat0\}$ is factorable since $\chi_{\{\hat0\}} (t)=1$. 
In this case, $\exp(\{\hat0\}) =\emptyset$.

Let $\Pc$ and $ \Qc$ be posets. 
A \emph{poset morphism} $\sigma : \Pc \to \Qc$ is an order-preserving map, i.e., $x \le y$ implies $\sigma(x) \le \sigma(y)$ for all  $x, y \in \Pc$. 
We call $\sigma$ a \emph{poset isomorphism} if $\sigma$ is bijective and its inverse is a poset morphism. 
The posets $\Pc$ and $ \Qc$ are said to be \emph{isomorphic}, written $\Pc \simeq \Qc$ if there exists a poset isomorphism  $\sigma : \Pc \to \Qc$.

For a subset $T \subseteq \Pc$,  the \emph{join} $ \bigvee T$ (resp., \emph{meet} $\bigwedge T$) of $T$ is the set of minimal upper bounds  (resp., maximal lower bounds) of elements in $T$. That is,
$$ \bigvee T := \min \{  b \in \Pc \mid  b \ge a, \,\forall a \in T\} \quad  \mbox{and}  \quad \bigwedge T := \max \{  b \in \Pc \mid  b \le a,\, \forall a \in T\} .$$
In particular, when $ T = \{x, y\}$, we write $x \vee y :=  \bigvee T$ and $x \wedge y := \bigwedge T$.

For $x \in \Pc$, define
$$\Pc_{\le x} : = \{ y \in \Pc \mid y \le x\} \quad  \mbox{and}  \quad \Pc_{\ge x} : = \{ y \in \Pc \mid y \ge x\}.$$

We call $x \in \Pc$ an \emph{atom} if $\rk(x)=1$. 
Denote the set of atoms of $\Pc$ by $A(\Pc)$. 
For $x, y \in \Pc$, by $y$ \emph{covers} $x$, written $x \cov y$, we mean $x<y$ and $x\le z<y$ implies $x = z$.

The poset $\Pc$ is a \emph{lattice} if $|x \vee y| = 1$ and $| x \wedge y| = 1$ for any $x, y \in \Pc$. 
In this case by abuse of notation we write, e.g., $a = x \vee y$ for $a \in x \vee y$.

\begin{definition} 
\label{def:GL}
A lattice $L$ is called \emph{geometric} if for all $x, y \in L$:   $x \cov y$  if and only if there is an atom $a \in A(L)$ with $a \not\le x$, $y = x \vee a$.
\end{definition}

\begin{definition} 
\label{def:LGP}
A poset $\Pc$ is called \emph{locally geometric} if $\Pc_{\le x}$ is a geometric lattice for every $x \in \Pc$.
\end{definition}

\begin{remark} 
\label{rem:LGP}
If $\Pc$ is a locally geometric poset, then so are $\Pc_{\le x}$ and $\Pc_{\ge x}$ for any $x \in \Pc$
\cite[Remark 2.2.6]{BD22}.
\end{remark}

\begin{definition} 
\label{def:gen-subposet}
 For any subset $B \subseteq A(\Pc)$, define $\Pc(B)$ to be the poset consisting of the minimal element $\hat0$ and all possible joins of the elements in $B$. 
We call $\Pc(B)$ the \emph{subposet of $\Pc$ generated by $B$}.
\end{definition}

\begin{remark} 
\label{rem:generated}
Note that $\Pc(A(\Pc)) =\Pc$ and every element of $\Pc(B)$ is an element of $\Pc$. 
If $\Pc$ is a locally geometric poset (or a lattice), then so is $\Pc(B)$. 
\end{remark}

\begin{definition} 
\label{def:modular}
 An element $x$ in a geometric lattice $L$ is \emph{modular} if for all $z\le x$ and all $y\in L$:
 $$x \wedge (y \vee z) =  (x \wedge  y) \vee z .$$
\end{definition}

Let $\Pc$ be a locally geometric poset. 
An \emph{order ideal} in $\Pc$ is a downward-closed subset. 
The poset $\Pc$ (or an order ideal of $\Pc$) is  called \emph{pure}  if all maximal elements have the same rank. 
An order ideal $\Qc$ of $\Pc$ is \emph{join-closed} if $T \subseteq \Qc$ implies $\bigvee T  \subseteq \Qc$. 
We denote by $\max(\Pc)$ the set of maximal elements in $\Pc$.

 \begin{definition}[{\cite[Definitions 2.4.1 and 5.1.1]{BD22}}] 
\label{def:M-TM-ideal}
An \emph{M-ideal} of a locally geometric poset $\Pc$ is a pure, join-closed, order ideal $\Qc \subseteq \Pc$ satisfying the following two conditions:
\begin{enumerate}[(1)]
\item $|a\vee y|\ge1$ for any $y\in \Qc $ and $a \in A(\Pc)  \setminus A (\Qc)$,
\item  for every $x \in \max(\Pc)$, there is some $y \in \max(\Qc)$ such that $y$ is a modular element in the
geometric lattice $ \Pc_{\le x}$.
\end{enumerate}
An M-ideal $\Qc \subseteq \Pc$ is called a  \emph{TM-ideal} if condition (1) above is replaced by a stronger condition that such $a$ and $y$ have a unique minimal upper bound, i.e.,
\begin{enumerate}[(1*)]
\item $|a\vee y|=1$ for any $y\in \Qc $ and $a \in A(\Pc)  \setminus A (\Qc)$.
\end{enumerate}
\end{definition}

Note that the element $y$ in Definition \ref{def:M-TM-ideal}(2) is necessarily unique since $\Qc$ is  join-closed. 
The following is a generalization of Stanley's supersolvable lattices \cite{St72}. 

\begin{definition}[{\cite[Definitions 2.5.1 and 5.1.4]{BD22}}] 
\label{def:SSS}
A locally geometric poset $\Pc$ is \emph{supersolvable} (resp., \emph{strictly supersolvable)} if there is a chain, called  an \emph{M-chain} (resp., a \emph{TM-chain)}
$$\{\hat 0\} = \Qc_0 \subsetneq \Qc_1 \subsetneq \cdots \subsetneq  \Qc_r =\Pc,$$
where each $\Qc_i$ is an M-ideal  (resp., a TM-ideal)  of $\Qc_{i+1}$ with $\rk(\Qc_i) = i$. 
\end{definition}

   \begin{theorem}[{\cite[Theorem 5.2.1]{BD22}}] 
\label{thm:SSS-factor}
 Let $\Qc$ be a TM-ideal of a locally geometric poset $\Pc$ with $\rk(\Qc) = \rk(\Pc) - 1$, and let $d =|A(\Pc)  \setminus A (\Qc)|$. Then
$$\chi_{\Pc}(t)  = (t-d)\chi_{\Qc}(t).$$
In particular, if $\Pc$ is strictly supersolvable with a TM-chain $\{\hat 0\} = \Qc_0 \subsetneq \Qc_1 \subsetneq \cdots \subsetneq  \Qc_r =\Pc$, and $d_i=|A(\Qc_i)  \setminus A (\Qc_{i-1})|$ for each $i$, then $\Pc$ is factorable with exponents 
 $$\exp(\Pc) = \{d_1,\ldots,d_r\}.$$
\end{theorem}

\begin{definition} 
\label{def:LSS}
A locally geometric poset $\Pc$ is \emph{locally supersolvable} if $\Pc_{\le x}$ is supersolvable  for every $x \in \Pc$.
\end{definition}

\begin{remark} 
\label{rem:S}
Denote by $\mathbf{SSS}$, $\mathbf{SS}$ and $\mathbf{LSS}$ the class of strictly supersolvable, supersolvable and locally supersolvable posets, repecstively. 
By \cite[Remark 2.5.4 and Example 5.2.5]{BD22},
 $$\mathbf{SSS} \subsetneq \mathbf{SS}  \subsetneq \mathbf{LSS}.$$
 Moreover, if $L$ is a geometric lattice, then $L \in \mathbf{SSS}$ if and only if $L\in \mathbf{SS}$ \cite[Proposition 5.1.9]{BD22}.
\end{remark}

\begin{definition}[{\cite[Definition 4.1.1]{BD22}}] 
\label{def:GP}
A locally geometric poset $\Pc$ is \emph{geometric} if for all $x, y \in \Pc$:
 if $\rk(x)<\rk(y)$ and $ I \subseteq A(\Pc)$ is such that $y \in \bigvee I $ and $|I|=\rk(y)$, then there is $a \in I$ such that $a \not\le x$ and $a \vee x \ne \emptyset$.
\end{definition}

When a poset is geometric, we have the following useful characterization of an M-ideal.
 \begin{lemma}[{\cite[Theorem 4.1.2]{BD22}}] 
\label{lem:GPMI}
Let $\Pc$ be a geometric poset, and let $\Qc$ be a pure, join-closed, proper order ideal of $\Pc$.
Then $\Qc$ is an M-ideal with $\rk(\Qc) = \rk(\Pc) - 1$ if and only if for any two distinct $a_1,a_2 \in A(\Pc)\setminus A(\Qc)$ and every $x \in a_1\vee a_2$ there exists $a_3 \in A(\Qc)$ such that $x>a_3$.
 \end{lemma}
 
\subsection{Free arrangements}
\label{subsec:free}
Now we recall the definition of free arrangements and their related properties. Our standard reference is \cite{OT92}. 
Throughout this subsection, an ``arrangement" means a ``central hyperplane arrangement".

Let $ \mathbb{K} $ be a field and let $T=  \mathbb{K}^{\ell} $. 
Let  $ \Hc $ be an arrangement in $T$. 
Let $ L(\Hc) $ be the intersection poset of $ \Hc $. 
We agree that $ T $ is a unique minimal element in $ L(\Hc) $. 
Thus $ L(\Hc) $ is a geometric lattice which can be equipped with the rank function $ \rk(X) :=\operatorname{codim}(X) $ for $X \in L(\Hc)$ (e.g., \cite[Lemma 2.3]{OT92}). 
We also define the \emph{rank} $\rk(\Hc)$ of $\Hc$ as the rank of the maximal element of $L(\Hc)$.

The \emph{characteristic polynomial} $ \chi_{\Hc}(t)  $ of $ \Hc $ is defined by
\begin{align*}
\chi_{\Hc}(t) :=t^{\ell - \rk(\Hc)} \cdot \chi_{L(\Hc)}(t), 
\end{align*}
where $ \chi_{L(\Hc)}(t)$ is the characteristic polynomial of the lattice $L(\Hc)$ defined in the preceding subsection. 
Definition \ref{def:poset-exp} motivates the following concept. 

\begin{definition} 
\label{def:hyparr-exp}
An arrangement $\Hc$ is called \emph{factorable} if its intersection poset $L(\Hc)$ is factorable (Definition \ref{def:poset-exp}).
In this case, we also call the roots of $ \chi_{\Hc}(t)  $ the \emph{(combinatorial) exponents} of $\Hc$ and use the notation $\exp(\Hc)$ to denote the multiset of exponents. 
Denote also by $\mathbf{FR}$ the class of factorable arrangements.
\end{definition} 

\vn
\textbf{Notation.}
 If an element $e$ appears $d\ge0$ times  in a multiset $M$, we  write $e^d \in M$. 
 
If $\Hc \in \mathbf{FR}$, then 
$$ \exp(\Hc) = \{0^{\ell - \rk(\Hc)} \} \cup \exp(L(\Hc)) .$$

The \emph{empty arrangement} $\varnothing_\ell$ (or simply $\varnothing$) is the arrangement in $T$ consisting of no elements. 
 In particular, $ \varnothing_\ell\in \mathbf{FR}$ with $ \exp(\varnothing_\ell) = \{0^{\ell} \}.$

Let $ \{x_{1}, \dots, x_{\ell}\} $ be a basis for the dual space $T^* $ and let $ S :=\mathbb{K}[x_{1}, \dots, x_{\ell}] $. 
For each $H\in \Hc$, fix a \emph{defining polynomial} $ \alpha_H=a_1x_1+\cdots+a_\ell x_\ell \in T^*$  $(a_i \in \mathbb{K})$ of  $H$, i.e., $H = \ker \alpha_H$.

A $\mathbb{K}$-linear map $\theta:S\to S$ is called a \emph{derivation} if $\theta(fg) = \theta(f)g + f\theta(g)$ for all $f,g \in S$.
Let $\Der(S)$ be the set of all derivations of $S$. 
It is a free $S$-module with a basis $\{\partial/\partial x_1,\ldots,\partial/\partial x_{\ell}\}$ consisting of the usual partial derivatives.
We say that a nonzero derivation $\theta  = \sum_{i=1}^\ell f_i \partial/\partial x_{i}$  is \emph{homogeneous of degree} $p$ if each nonzero coefficient $f_i$ is a homogeneous polynomial of degree $p$ \cite[Definition 4.2]{OT92}.

The concept of free arrangements was defined by Terao  \cite{T80, OT92}. 
\begin{definition}[{\cite[Definitions 4.5 and 4.15]{OT92}}]\label{def:free-arr}
The \emph{module $D(\Hc) $ of logarithmic derivations}  is defined by 
	\begin{equation*}
	D(\Hc) := \{ \theta \in \Der(S) \mid \theta( \alpha_H) \in  \alpha_HS \mbox{ for all } H\in \Hc\}.
	\end{equation*}
	We say that $\Hc$ is \emph{free} if the module $D(\Hc)$ is a free $S$-module. 
	Denote by $\mathbf{F}$ the class of free arrangements.
\end{definition}

If $\Hc\in \mathbf{F}$, we may choose a basis $\{\theta_1, \ldots, \theta_\ell\}$ consisting of homogeneous derivations for $D(\Hc)$ \cite[Proposition 4.18]{OT92}. 
Although a basis is not unique, the degrees of the derivations in a basis are uniquely determined by $\Hc$ \cite[Proposition A.24]{OT92}.  

The following theorem of Terao connects algebraic and combinatorial properties of an arrangement.

\begin{theorem}[{\cite[Main Theorem]{T81}}, {\cite[Theorem 4.137]{OT92}}]\label{thm:Factorization}
If $\Hc$ is free, then $\Hc$ is factorable with combinatorial exponents given by the degrees of the elements in any basis for $D(\Hc)$. 
\end{theorem}

Based on this, Terao conjectured that  freeness is a  combinatorial property \cite[Conjecture 4.138]{OT92}. 
Although Terao's conjecture is still open, there are some subclasses of free arrangements that are known to be combinatorially determined. 

\begin{definition} 
\label{def:SS-arr}
An arrangement $\Hc$ is called \emph{supersolvable} if its intersection lattice $L(\Hc)$ is supersolvable (Definition \ref{def:SSS}).
Denote also by $\mathbf{SS}$ the class of supersolvable ($=$ strictly supersolvable) central hyperplane arrangements.
\end{definition}

Fix $H\in\Hc$, define the \emph{deletion} $\Hc':=\Hc \setminus \{H\}$ and \emph{restriction} $\Hc'' := \{ H \cap K \mid K \in \Hc'\}.$ 
Then $\Hc'$ is an arrangement in $V$, and $\Hc''$ is an arrangement in $H \simeq\mathbb{K}^{\ell-1} $.
 
\begin{definition}[{\cite[Definition 4.53]{OT92}}] 
\label{def:IF-arr}
The class $\mathbf{IF}$ of \emph{inductively free arrangements} is the smallest class of arrangements which satisfies
\begin{enumerate}[(1)]
\item $\varnothing_\ell\in\mathbf{IF}$ for $\ell \ge 1$,
\item $\Hc \in \mathbf{IF}$ if there exists $H \in \Hc$ such that $\Hc'' \in \mathbf{IF}$, $\Hc' \in \mathbf{IF}$, and $\chi_{\Hc''}(t) $ divides $\chi_{\Hc'}(t)$.
\end{enumerate}
\end{definition}

\begin{definition}[{\cite[Theorem--Definition 4.3]{Abe16}}] 
\label{def:DF-arr}
The class $\mathbf{DF}$ of \emph{divisionally free arrangements} is the smallest class of  arrangements which satisfies
\begin{enumerate}[(1)]
\item $\varnothing_\ell\in\mathbf{DF}$ for $\ell \ge 1$,
\item $\Hc \in \mathbf{DF}$ if there exists $H \in \Hc$ such that $\Hc'' \in \mathbf{DF}$ and $\chi_{\Hc''}(t) $ divides $\chi_{\Hc}(t)$.
\end{enumerate}
\end{definition}

\begin{remark} 
\label{rem:all-property}
Supersolvability, inductive and divisional freeness of central hyperplane arrangements all are combinatorial properties. 
We give below the relation between the concepts we have defined so far:
$$ \tbf{SSS} = \tbf{SS}  \subsetneq   \tbf{IF} \subsetneq \tbf{DF}  \subsetneq \tbf{F}  \subsetneq \tbf{FR}.$$

The first containment is proved by Jambu and Terao \cite[Theorem 4.2]{JT84}. 
The arrangement of a root system of type $D_\ell$ for $\ell \ge 4$ belongs to $ \tbf{IF} \setminus \tbf{SS}$ (e.g.,  \cite[Theorem 6.6]{Hul16}). 
The second  containment follows from the deletion-restriction formula  $\chi_{\Hc}(t)=
\chi_{\Hc'}(t)-\chi_{\Hc''}(t)$  (e.g., \cite[Theorem 2.56]{OT92}). 
The arrangement defined by the exceptional complex reflection group of type $G_{31}$ is known to be divisionally free \cite[Theorem 1.6]{Abe16} but not inductively free \cite[Theorem 1.1]{HR15}. 
The third containment is proved by Abe \cite[Theorem 1.1]{Abe16}. 
The \emph{intermediate arrangement} $\A^0_\ell(r)$ for $\ell \ge 3$, $r \ge 3$ in \cite[Theorem 5.6]{Abe16} is an example of an arrangement in $ \tbf{F} \setminus \tbf{DF}$. 
The fourth containment is Theorem \ref{thm:Factorization} by Terao. 
There are many examples of factorable but not free arrangement, e.g., \cite[3.6]{FR86}.

\end{remark}

\section{Inductive and divisional posets} 
\label{sec:ind-poset}
From now on unless otherwise stated, we will assume that $\Pc$ is a locally geometric poset, and set $A= A(\Pc)$ and $r =\rk(\Pc)$.

\begin{definition}
\label{def:triple-posets}
Fix an atom $a\in A$. Let $\Pc' :=\Pc(A\setminus\{a\})$ be the subposet of $\Pc$ generated by $A\setminus\{a\}$ and define $\Pc'':=\Pc_{\ge a}$. 
We call $(\Pc, \Pc', \Pc'')$ the triple of posets with distinguished atom $a$.
\end{definition}

\begin{remark} 
\label{rem:separator}
Note that for each $a\in A$, we have $\rk(\Pc) = \rk(\Pc')+\epsilon(a)$, where $\epsilon(a)$ is either $0$ or $1$. 
Indeed, let $x \in \max(\Pc)$ so that $\rk(x)=r$. If $a \not\le x$ then $ \rk(\Pc')=r$. 
Otherwise, set $\Qc := \Pc_{\le x}$ then $a \in A(\Qc)$. 
Let $(\Qc, \Qc', \Qc'')$ the triple of posets with distinguished atom $a$. 
Since $\Qc$ is a geometric lattice with $\rk(\Qc) =r$, it follows that  $\rk(\Qc') \le r \le \rk(\Qc')+1$.
Note that $\Qc'$ is a subposet of $\Pc'$. 
Then $r \ge \rk(\Pc') \ge \rk(\Qc') \ge r-1$, as desired. 

We call $a\in A$ a \emph{separator} of $\Pc$ if $\epsilon(a) =1$. 
\end{remark}

For each $x \in \Pc$, define
$$A_x := \{ a \in A \mid a \le x\}.$$

 \begin{lemma}[{\cite[Lemma 2.35]{OT92}}] 
\label{lem:crosscut}
Let $\Pc$ be a geometric lattice. 
For $x, y \in \Pc$ with $x \le y$, let $S(x,y)$ be the set of all subsets $B \subseteq A$ such that $A_x \subseteq B$ and $\max(\Pc(B))=y$. 
Then 
$$\mu(x,y) = \sum_{B \in S(x,y)} (-1)^{|B \setminus A_x|}.$$
 \end{lemma}
 
  \begin{lemma} 
\label{lem:SAIS}
Let $\Pc$ be a locally geometric poset. Then the characteristic polynomial $\chi_{\Pc}(t)$ strictly alternates in sign, i.e., if
$$\chi_{\Pc}(t)=c_rt^r +c_{r-1}t^{r-1}+\cdots+c_0, $$
then $(-1)^{r-i}c_i >0$ for $0\le i \le r$.
 \end{lemma}
 
   \begin{proof}
By definition, for each $0\le i \le r$ we have
$$(-1)^{r-i}c_i = \sum_{\rk(x) \,=\, r-i} (-1)^{r- i} \mu(\hat0, x).$$
Note that  the characteristic polynomial of  a geometric lattice strictly alternates in sign (e.g., \cite[Corollary 3.5]{St07}). 
Thus $(-1)^{\rk(x)} \mu(\hat0, x) >0$ since $\Pc_{\le x}$ is a geometric lattice  for every $x \in \Pc$. 
Hence $(-1)^{r-i}c_i >0$ for each $0\le i \le r$.
\end{proof}
 
We show below that the characteristic polynomials of locally geometric posets satisfy a deletion-restriction recurrence, which is crucial for our subsequent discussion. 
This formula is already proved for geometric lattices, e.g., see \cite[Theorem 1.2.20]{Bra92}. 
The method therein can be readily extended to locally geometric posets, we include here a proof for the sake of completeness. 
   \begin{theorem}
\label{thm:LGP-DC}
Let $\Pc$ be a locally geometric poset and fix  $a\in A$. 
Then 
\begin{equation*}
\chi_{\Pc}(t)=
t^{\epsilon(a)}\cdot\chi_{\Pc'}(t)-
\chi_{\Pc''}(t). 
\end{equation*}
Here $\epsilon(a) = \rk(\Pc) - \rk(\Pc')$ is either $0$ or $1$ by Remark \ref{rem:separator}. 
\end{theorem}

  \begin{proof}
  Since $\Pc_{\le x}$ is a geometric lattice  for every $x \in \Pc$, by Lemma \ref{lem:crosscut} we have
\begin{align*}
\chi_{\Pc}(t) &=  \sum_{x \in \Pc} 
 \sum_{\substack{B \subseteq A_x \\  x\,=\, \max(\Pc(B))}} (-1)^{|B|}   t^{r-\rk(x)} \\
 & =  \sum_{x \in \Pc} 
 \sum_{\substack{ a \notin B \subseteq A_x \\  x\,=\, \max(\Pc(B))}} (-1)^{|B|}   t^{r-\rk(x)} 
 +
  \sum_{x \in \Pc} 
 \sum_{\substack{a \in B \subseteq A_x \\  x\,=\, \max(\Pc(B))}} (-1)^{|B|}  t^{r-\rk(x)} \\
 & =  \sum_{x \in \Pc'} 
 \sum_{\substack{ B \subseteq A_x \\  x\,=\, \max(\Pc(B))}} (-1)^{|B|}  t^{\rk(\Pc')+\epsilon(a)-\rk(x)} 
-
  \sum_{x \in \Pc_{\ge a}} 
 \sum_{B \in S(a,x)} (-1)^{|B \setminus A_a|} t^{r-\rk(x)} \\
  & =  t^{\epsilon(a)}\cdot\chi_{\Pc'}(t)
-
  \sum_{x \in \Pc''} \mu(a,x) t^{\rk''(\Pc'')-\rk''(x)} \\
    & =  t^{\epsilon(a)}\cdot\chi_{\Pc'}(t) - \chi_{\Pc''}(t). \qedhere
\end{align*}
\end{proof}

Now we introduce the protagonists of the paper.

\begin{definition} 
\label{def:IP}
The class $\mathbf{IP}$ of \emph{inductive posets} is the smallest class of locally geometric posets which satisfies
\begin{enumerate}[(1)]
\item  $\{\hat0\}\in\mathbf{IP}$,
\item $\Pc \in \mathbf{IP}$ if there exists an atom $a \in A$ such that $\Pc'' \in \mathbf{IP}$, $\Pc' \in \mathbf{IP}$, and $\chi_{\Pc''}(t)$ divides $\chi_{\Pc'}(t)$.
\end{enumerate}
\end{definition}

\begin{definition} 
\label{def:DP}
The class $\mathbf{DP}$ of \emph{divisional posets} is the smallest class of locally geometric posets which satisfies
\begin{enumerate}[(1)]
\item $\{\hat0\}\in\mathbf{DP}$,
\item $\Pc \in \mathbf{DP}$ if there exists an atom $a \in A$ such that $\Pc'' \in \mathbf{DP}$ and $\chi_{\Pc''}(t)$ divides $\chi_{\Pc}(t)$.
\end{enumerate}
\end{definition}

Here are the first two important properties of the inductive and divisional posets.

\begin{proposition}
\label{prop:IP<DP}
If $\Pc \in \mathbf{IP}$ then $\Pc \in \mathbf{DP}$.
\end{proposition}

  \begin{proof}
  We argue by induction on  $r=\rk(\Pc) \ge 0$. 
The assertion clearly holds true when $r= 0$.
Suppose $r > 0$. 
Since $\Pc \in \mathbf{IP}$, there exists an atom $a \in A$ such that $\Pc'' \in \mathbf{IP}$ and $\chi_{\Pc''}(t)$ divides $\chi_{\Pc'}(t)$. 
By the induction hypothesis, $\Pc'' \in \mathbf{DP}$. 
Furthermore, by Theorem \ref{thm:LGP-DC}, $\chi_{\Pc''}(t)$ divides $\chi_{\Pc}(t)$. 
(Note that $t \nmid  \chi_{\Pc''}(t)$ by Lemma \ref{lem:SAIS}.)
Thus $\Pc \in \mathbf{DP}$ as desired.
\end{proof}

\begin{proposition}
\label{prop:IP-DP-preserved}
Let $\Pc, \Qc$ be two isomorphic locally geometric posets. 
Then $\Pc \in \mathbf{IP}$ (resp., $\Pc \in \mathbf{DP}$) if and only if  $\Qc \in \mathbf{IP}$ (resp., $\Qc \in \mathbf{DP}$).
\end{proposition}

  \begin{proof}
We show the assertion for  $\mathbf{IP}$ by double induction on the rank $r$ and number $|A|$ of atoms. 
The assertion for $\mathbf{DP}$ can be proved by induction on the rank $r$ by a similar (and easier) argument. 

The assertion is clearly true when $r=0$ or $|A|=0$. 
Suppose $r \ge 1$ and $|A|\ge 1$.
Let $f: \Pc \to \Qc$ be a poset isomorphism. 
Suppose $\Pc \in \mathbf{IP}$. 
Then there exists an atom $a \in A$ such that $\Pc'' \in \mathbf{IP}$, $\Pc' \in \mathbf{IP}$, and $\chi_{\Pc''}(t)$ divides $\chi_{\Pc'}(t)$. 
Define $\Qc' :=\Qc(A(\Qc)\setminus\{f(a)\})$ and $\Qc'':=\Qc_{\ge f(a)}$. 
Hence $\Pc' \simeq \Qc'$ and $\Pc'' \simeq \Qc''$.
Note that $|A(\Pc')| < |A(\Pc)|$ and $\rk''(\Pc'') < \rk(\Pc)$. 
By the induction hypothesis, $\Qc'' \in \mathbf{IP}$ and $\Qc' \in \mathbf{IP}$. 
It is also clear that $\chi_{\Qc''}(t)$ divides $\chi_{\Qc'}(t)$ since the characteristic polynomial is preserved under isomorphism. 

\end{proof}

\begin{remark} 
\label{rem:IL}
We address here some remarks about the relation of our inductive and divisional posets with some known concepts in literature. 
\begin{enumerate}[(1)]
	\item Brandt \cite[Definition 1.2.21]{Bra92} defined the class $\mathbf{IL}$ of \emph{inductive lattices} to be the smallest class of geometric lattices which satisfies: (1) $\{\hat0\} \in \mathbf{IL}$ and (2) $\Pc \in \mathbf{IL}$ if there exists an atom $a \in A$ such that $\Pc'' \in \mathbf{IL}$, $\Pc' \in \mathbf{IL}$, and $\chi_{\Pc''}(t)$ divides $\chi_{\Pc'}(t)$. 
Thus for a geometric lattice $\Pc$, we have that $\Pc \in \mathbf{IL}$ if and only if $\Pc \in \mathbf{IP}$.
	\item  A central hyperplane arrangement $\Hc$ in $V=\mathbb{K}^\ell$ is inductively free (resp., divisionally free) in Definition \ref{def:IF-arr} (resp., \ref{def:DF-arr})  if and only if the (geometric) intersection  lattice  $L(\Hc)$  of $\Hc$ is  inductive (resp., divisional). 
 In particular, $\tbf{IP} \subsetneq \tbf{DP} $ which follows from Remark \ref{rem:all-property}.
\end{enumerate}
\end{remark}

Now we give a proof of the first main result of the paper.

\begin{proof}[\tbf{Proof of  Theorem \ref{thm:DP-factor-intro}}]
We need to show that if $\Pc \in \mathbf{DP}$ with $r=\rk(\Pc)\ge 1$, then there are positive integers $d_1,\ldots,d_r\in \Z_{>0}$ such that 
$$\chi_{\Pc}(t) = \prod_{i=1}^{r}(t-d_i).$$

We argue by induction on  $r$. 
If $r= 1$ then $\chi_{\Pc}(t) =t- |A|$. 
The assertion clearly holds.
Suppose $r > 1$. 
Since $\Pc \in \mathbf{DP}$, there exists an atom $a \in A$ such that $\Pc'' \in \mathbf{DP}$ and $\chi_{\Pc''}(t)$ divides $\chi_{\Pc}(t)$.
By the induction hypothesis, there exist positive integers $d_1,\ldots,d_{r-1} \in \Z_{>0}$ and an integer $d_r \in \Z$ such that 
\begin{align*}
\chi_{\Pc''}(t)
& =\prod_{i=1}^{r-1}(t-d_i),  \\
\chi_{\Pc}(t) 
& = (t-d_r)\chi_{\Pc''}(t).
\end{align*}
Moreover, $d_1d_2\cdots d_r>0$ by Lemma \ref{lem:SAIS}. 
Thus $d_r>0$.
\end{proof}

Thus the divisionality of a poset is a sufficient condition for its factorability.
The following necessary and sufficient condition for a poset to be divisional is immediate from Definition \ref{def:DP}. 
Note that the sum of all exponents of a divisional poset equals the number of atoms. 

   \begin{theorem}
\label{thm:div-chain}
A locally geometric poset  $\Pc$ of rank $r$ is divisional if and only if there exists a chain, called a \emph{divisional chain}
$$\hat 0 = x_0 < x_1 < \cdots <  x_r,$$
such that $\rk(x_i) = i$ and $\chi_{\Qc_i}(t)$ divides $\chi_{\Qc_{i-1}}(t)$ where $\Qc_i := \Pc_{\ge x_{i}}$ for each $1 \le i \le r$. 
In this case, 
 $\exp(\Pc) = \{d_1,\ldots,d_r\}$
 where $d_i : = |A(\Qc_{i-1})| -  |A(\Qc_i)| $.
\end{theorem}

\begin{remark} 
\label{rem:converse}
The converse of Theorem \ref{thm:DP-factor-intro} is not true in general. 
Namely, there exists a factorable poset that is not divisional. 
An example from hyperplane arrangements is already mentioned in Remark \ref{rem:all-property}. 
We give here an example of a poset that is not a lattice. 
In \cite[Example 4.6]{Hal17}, the \emph{weighted partition poset} $\Pc : =\Pi_3^w$ of rank $3$ is given with the characteristic polynomial 
$\chi_{\Pc}(t)  = (t-3)^2$ (see Figure \ref{fig:pi3w}). 
However, $\Pc$ is not divisional because $\chi_{ \Pc_{\ge x}}(t)=t-2$ does not divide   $\chi_{\Pc}(t)$ for any atom $x$.
\end{remark}

	 	\begin{figure}[htbp]
\centering
\begin{tikzpicture}[scale=.7]
\draw (0,0) node[v](0){};
\draw (1,2) node[v](4){};
\draw (3,2) node[v](5){};
\draw (5,2) node[v](6){};
\draw (-1,2) node[v](3){};
\draw (-3,2) node[v](2){};
\draw (-5,2) node[v](1){};
\draw (-4,4) node[v](7){};
\draw (0,4) node[v](8){};
\draw (4,4) node[v](9){};

\draw (0)--(1)--(7)--(2)--(0)--(3)--(7);
\draw (0)--(4)--(9)--(5)--(0)--(6)--(9);
\draw (1)--(8)--(2);
\draw (3)--(8)--(4);
\draw (5)--(8)--(6);
\end{tikzpicture}
\caption{The weighted partition poset $\Pi_3^w$.}
\label{fig:pi3w}
\end{figure}
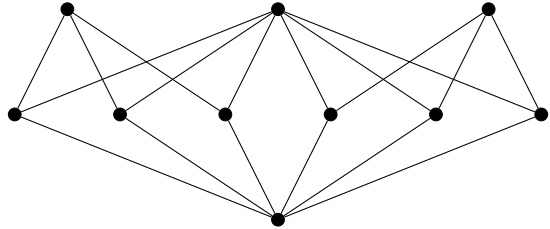

By Proposition \ref{prop:IP<DP}, the exponents of an inductive poset are defined naturally. 
The following ``addition" theorem for inductive posets follows readily from Definition \ref{def:IP} and Theorem \ref{thm:LGP-DC}.

   \begin{theorem}
\label{thm:add-IP}
Let $\Pc$ be a locally geometric poset with $A \ne \emptyset$ and let $a \in A$. 
\begin{enumerate}[(a)]
\item Suppose that $a$ is not a separator of $\Pc$. 
If $\Pc'' \in \mathbf{IP}$ with $\exp(\Pc'') = \{d_1,\ldots,d_{\ell-1}\}$ and $\Pc' \in  \mathbf{IP}$ with $\exp(\Pc') = \{d_1,\ldots,d_{\ell-1},d_\ell\}$, then $\Pc \in \mathbf{IP}$ with $\exp(\Pc) = \{d_1,\ldots,d_{\ell-1},d_\ell+1\}$.
\item Suppose that $a$ is a separator of $\Pc$. 
If $\Pc'' \in \mathbf{IP}$, $\Pc' \in  \mathbf{IP}$ with $\exp(\Pc'') =\exp(\Pc') = \{d_1,\ldots,d_{\ell-1}\}$, then $\Pc \in \mathbf{IP}$ with $\exp(\Pc) = \{1,d_1,\ldots,d_{\ell-1}\}$.
\end{enumerate}
\end{theorem}

The process of constructing an inductive poset $\Pc$ from the trivial lattice (or more generally, from an inductive subposet generated by some atoms) by adding an atom one at a time with the aid of Theorem  \ref{thm:add-IP} is called an \emph{induction table}. 
Each row of the table records the exponents of $\Pc'$ and $\Pc''$ and the atom $a$ added at each step.
The last row displays the exponents of $\Pc$. 

We will see in Section \ref{sec:rootsystem} many examples of posets which are both inductive and geometric arising from abelian arrangements.
Figure \ref{fig:ind-not-geo} below depicts an inductive poset that is not geometric. 
(In particular, it is not the poset of layers of an abelian arrangement by Theorem \ref{thm:GP}.)

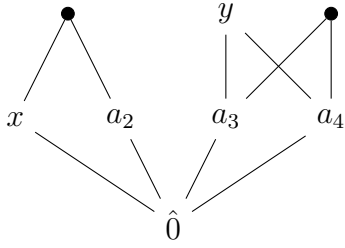
\begin{figure}[htbp]
\centering
\begin{subfigure}{.35\textwidth}
  \centering
\begin{tikzpicture}[scale=.7]
  \node (0) at (0,0) {$\hat 0$};
  \node (2) at (-3,2) {$x$};
    \node (3) at (-1,2) {$a_2$};
      \node (4) at (1,2) {$a_3$};
        \node (5) at  (3,2)  {$a_4$};
                \node (9) at  (1,4)  {$y$};
\draw (-2,4) node[v](8){};
\draw (3,4) node[v](10){};
\draw (0)--(2)--(8)--(3)--(0)--(4)--(9)--(5)--(0);
\draw (5)--(10)--(4);
\end{tikzpicture}
\label{fig:ING-poset}
\end{subfigure}%
\qquad
\begin{subfigure}{.45\textwidth}
  \centering
    {\footnotesize\renewcommand\arraystretch{1.5} 
\begin{tabular}{ccc}
\hline
$\exp(\Pc')$ &  $a$ &   $\exp(\Pc'')$  \\
\hline
$\emptyset$ &  $x$ & $\emptyset$ \\
$1$ &  $a_3$ & $\emptyset$ \\
$2$ &  $a_4$ & $2$ \\
$1,2$ &  $a_2$ & $1$ \\
$1,3$ &    &   \\
\hline
\end{tabular}
}
  \label{fig:ING-table}
\end{subfigure}
\caption{An inductive poset that is not geometric (left) and an induction table for its inductiveness (right). The elements labelled by $x$ and $y$ do not satisfy the requirement of Definition \ref{def:GP}.
}
\label{fig:ind-not-geo}
\end{figure}

\section{Strictly supersolvable implies inductive}
\label{sec:sss-ind} 

In this section we prove the second main result of the paper (Theorem \ref{thm:SSS<IP-intro}). 
First we need some basic facts of M-ideals. 
All posets in this section are locally geometric.

 \begin{lemma} 
\label{lem:pure}
If a poset $\Pc$ has an M-ideal $\Qc$ with $\rk(\Qc) = \rk(\Pc) - 1$, then $\Pc$ is necessarily pure. 
 \end{lemma}
 
   \begin{proof}
 First note that $ A(\Pc)  \setminus A (\Qc) \ne \emptyset$ since $\Qc$ is join-closed. 
Fix an arbitrary $x\in\max(\Pc)$. 
 If $x \in \Qc$, then by Condition \ref{def:M-TM-ideal}(1)  for any $a \in A(\Pc) \setminus A(\Qc)$ there exists $b \in a \vee x$ such that $x <b$, a contradiction. 
 We may assume $x \in \Pc  \setminus \Qc$. 
 Then by Condition \ref{def:M-TM-ideal}(2),  there exists $y\in\max(\Qc)$ such that $y <x$. 
 Thus $\rk(x) > \rk(\Qc)$ and hence $\rk(x)  = \rk(\Pc)$.
\end{proof}
 
 \begin{lemma}[{\cite[Lemma 2.4.6]{BD22}}]
\label{lem:hat0}
Let $\Qc $ be an M-ideal of a poset $\Pc$ with $\rk(\Qc) = \rk(\Pc) - 1$ and let $a\in \Pc$.  
Then $a \in A(\Pc)  \setminus A (\Qc)$ if and only if $y \wedge  a = \hat 0$ for all $y\in\max(\Qc)$.
 \end{lemma}
 
\begin{proposition}[{\cite[Proposition 2.4.7]{BD22}}]
\label{prop:x'}
Let $\Qc $ be an M-ideal of a poset $\Pc$ with $\rk(\Qc) = \rk(\Pc) - 1$. 
Fix $x \in \Pc   \setminus \Qc$ and let  $y$ be an element in $\max(\Pc)$ such that $x \le y$.  
Let $y'$ be the unique element in $\max(\Qc)$ such that ($y$ covers $y'$ and) $y'$ is a modular element in the
geometric lattice $\Pc_{\le y}$ (Definition \ref{def:M-TM-ideal}). 
Then $x':=y' \wedge x$ is the unique element in $\Qc$ such that $x$ covers $x'$ and $x'$ is modular in $ \Pc_{\le x}$.
\end{proposition}

Now we prove a new property of  a TM-ideal, extending a well-known property  \cite[Lemma 1]{St71} of a modular element in a finite geometric lattice.

 \begin{lemma}
\label{lem:>=a}
 If $\Qc$ is a TM-ideal of a poset $\Pc$ with $\rk(\Qc) = \rk(\Pc) - 1$, then for any $a \in A(\Pc)  \setminus A (\Qc)$ there is a poset isomorphism $\Qc \simeq \Pc_{\ge a}$.  
 \end{lemma}
 \begin{proof} 
Fix $a \in A(\Pc)  \setminus A (\Qc)$ and denote $\scR : = \Pc_{\ge a}$. 
 Owing to Definition \ref{def:M-TM-ideal}(1*) and Proposition \ref{prop:x'}, two poset maps $\sigma$ and $\tau$ below are well-defined:
$$
\sigma: \Qc \longrightarrow \scR  \mbox{ via } x \mapsto x \vee a,  \quad
\tau : \scR \longrightarrow \Qc  \mbox{ via } x \mapsto x'.
$$

We show that $\sigma$ is a poset isomorphism whose inverse is exactly $\tau$. 
First we show that both maps are order-preserving. 
The assertion for $\sigma$ is easy. 
To show the assertion for $\tau$ note that for $x_1 \le_\scR x_2 $, if $y\in\max(\Pc)$ and $x_2 \le_\scR  y$, then $\tau(x_1) = y' \wedge x_1$ and $\tau(x_2) = y' \wedge x_2$ where $y'$ is the unique element in $\max(\Qc)$ such that $y'$ is modular in $\Pc_{\le y}$. 
Thus  $\tau(x_1)\le_\Qc \tau(x_2)$ follows easily. 

Now we show $\sigma\circ \tau  = \tau \circ \sigma = \mathrm{id}$. 
If $x \in \scR$, then $(\sigma\circ \tau )(x)=\sigma(x') = x' \vee a =x$ where the last equality follows from Definition \ref{def:M-TM-ideal}(1*) since $x \in x' \vee a$. 

Let $x \in \Qc$, then $(\tau \circ \sigma)(x)=\tau(x \vee a) = (x \vee a)'$.  
It remains to show $(x \vee a)'=x$. 
If  $x $ and $(x \vee a)'$ are incomparable, then $x \vee a \in (x \vee a)' \vee x$ which contradicts the join-closedness of $\Qc$. 
Note that $\rk(x \vee a) > \rk(x)$ hence it cannot happen that $x > (x \vee a)'$. 
Thus we may assume $x \le (x \vee a)'$. 
Let $y\in\max(\Pc)$ so that $x \vee a \le   y$. 
Let $y'$ be the unique element in $\max(\Qc)$ such that $y'$ is modular in $\Pc_{\le y}$.
Then
$$ (x \vee a)' = y' \wedge (x \vee a) = x\vee (y' \wedge  a) = x \vee \hat 0 = x,$$
where the second equality follows from the modularity \ref{def:modular} of $y'$ in $\Pc_{\le y}$ with $x \le y'$, and the third equality follows from Lemma \ref{lem:hat0}. 
    \end{proof}

Using the lemma above, we show the following stronger version of Theorem \ref{thm:SSS<IP-intro}.

 \begin{lemma}
\label{lem:Q<P}
Let $\Qc$ be a TM-ideal of a poset $\Pc$ with $\rk(\Qc) = \rk(\Pc) - 1$. 
If $\Qc \in \mathbf{IP}$ (resp., $\Qc \in \mathbf{DP}$), then $\Pc \in \mathbf{IP}$  (resp., $\Pc \in \mathbf{DP}$) with 
$$\exp(\Pc) = \exp(\Qc) \cup \{|A(\Pc)  \setminus A (\Qc)|\}.$$
\end{lemma}

  \begin{proof}
  First we show the assertion for divisionality. 
  Fix $a \in A(\Pc)  \setminus A (\Qc)$. 
By Lemma \ref{lem:>=a}, $\Qc \simeq \Pc'' = \Pc_{\ge a}$. 
Suppose $\Qc \in \mathbf{DP}$. 
Then $\Pc'' \in \mathbf{DP}$ by Proposition \ref{prop:IP-DP-preserved}. 
Moreover, by  Theorem \ref{thm:SSS-factor}, 
$$\chi_{\Pc}(t)  = (t-m)\chi_{\Qc}(t),$$
where $m :=|A(\Pc)  \setminus A (\Qc)|$. 
Therefore, $\chi_{\Pc''}(t)$ divides $\chi_{\Pc}(t)$. 
Hence $\Pc \in \mathbf{DP}$ with $\exp(\Pc) = \exp(\Qc) \cup \{m\}$ as desired.
  
  Now we show the assertion for inductiveness by adding the atoms from $A(\Pc)  \setminus A (\Qc)$ to $A (\Qc)$ in any order successively with the aid of Theorem \ref{thm:add-IP}. 
  Write $A(\Pc)  \setminus A (\Qc) = \{a_1,\ldots,a_m\}$. 
  Let $A_i := A (\Qc) \cup  \{a_1,\ldots,a_i\}$ and $\Pc_i := \Pc(A_i)$ for each $1 \le i \le m$. 
  
  First note that by Lemma \ref{lem:pure}, the poset $\Pc$ is pure.
  We observe that $\rk(\Pc_i)= \rk(\Pc)=r$ for every $1 \le i \le m$. It is because $|a_i \vee y|=1$ and $\rk(a_i\vee y)=r$ for any $y\in \max(\Qc) $ and $a_i \in A_{i}  \setminus A(\Qc) \subseteq A  \setminus A(\Qc)$.
  
  We claim that $\Qc$ is a TM-ideal of rank $r-1$ of $\Pc_i$ for every $1 \le i \le m$. 
  (The case $i=m$ is obviously true.)
  Condition  \ref{def:M-TM-ideal}(1*) is clear. 
  It suffices to show   Condition  \ref{def:M-TM-ideal}(2).
  First consider $i=m-1$. 
  Fix $x \in \max(\Pc_{m-1}) \subseteq \max(\Pc)$.
  Denote $L := \Pc_{\le x}$ and  $L_{m-1} := (\Pc_{m-1})_{\le x}$. 
  Therefore $L$ and $L_{m -1}$ are geometric lattices sharing top element $x$. 
  We need to show that  there is some $y \in \max(\Qc)$ such that $y$ is a modular element in $L_{m-1}$. 
  Since $\Qc$ is a TM-ideal of $\Pc$, there exists $y' \in \max(\Qc)$ such that $y'$ is a modular element in $L$. 
  If $x \not> a_m$ then $L = L_{m -1}$. 
  We may take $y = y'$.
  If  $x > a_m$ then $L_{m -1} = L(A(L) \setminus \{a_{m}\}) $. 
  Since $y' \not> a_m$, we must have that $y' \in L_{m-1}$ and $y'$ is also a modular element in $L_{m-1}$ by \cite[Lemma 4.6]{JT84}. 
Again take $y = y'$.
  Use this argument repeatedly, we may show the claim holds true for every $1 \le i \le m-1$. 
  
  Now we show that $\Pc_i \in  \mathbf{IP}$ with $\exp(\Pc_i) = \exp(\Qc) \cup \{i\}$ for every $1 \le i \le m$. 
  Note that by Lemma \ref{lem:>=a}, $\Qc \simeq \Pc_{\ge a}$ for any $a \in A(\Pc)  \setminus A (\Qc)$. 
  It is not hard to check that $(\Pc_1, \Pc'_1 = \Qc, \Pc''_1 \simeq \Qc)$ is the triple of posets with distinguished atom $a_1$, and that $a_1$ is a separator of $\Pc_1$. 
  Hence $\Pc_1 \in  \mathbf{IP}$ with $\exp(\Pc_1) = \exp(\Qc) \cup \{1\}$ by Theorem \ref{thm:add-IP}. 
  Similarly, $(\Pc_2, \Pc'_2 = \Pc_1, \Pc''_2 \simeq \Qc)$ is the triple with distinguished atom $a_2$, and that $a_2$ is not a separator of $\Pc_2$. 
  Hence $\Pc_2 \in  \mathbf{IP}$ with $\exp(\Pc_2) = \exp(\Qc) \cup \{2\}$. 
  Use this argument repeatedly, we may show the claim holds true for every $1 \le i \le m$. 
  The case $i=m$ yields  $\Pc \in  \mathbf{IP}$ with $\exp(\Pc) = \exp(\Qc) \cup \{m\}$ as desired.

   \end{proof}

\begin{proof}[\tbf{Proof of  Theorem \ref{thm:SSS<IP-intro}}]
  Note that the trivial lattice is inductive. 
Apply Lemma \ref{lem:Q<P} repeatedly to the elements in any TM-chain of a strictly supersolvable poset $\Pc$. 
\end{proof}

\begin{example}\label{eg:Dowling}   
The \emph{Dowling posets} are proved to be strictly supersolvable \cite[Example 5.1.8]{BD22}. 
The poset of layers of the toric arrangement of an arbitrary ideal of a type $C$ root system with respect to the \emph{integer lattice} is also   strictly supersolvable (Theorem \ref{thm:CT-SSS}).
Hence these posets are inductive by Theorem \ref{thm:SSS<IP-intro}.
\end{example}

\begin{remark} 
\label{rem:JT-SSS<IP}
The main result of \cite{JT84} by Jambu and Terao mentioned in Remark \ref{rem:all-property} is a special case of our Theorem \ref{thm:SSS<IP-intro} when the poset is a geometric lattice. 
An induction table for a strictly supersolvable poset can easily  be constructed using the argument in the proof of Lemma \ref{lem:Q<P}.

The converse of Theorem \ref{thm:SSS<IP-intro} is not true in general.
There are many known examples of central hyperplane arrangements whose intersection lattices are inductive but not (strictly) supersolvable (see e.g., Theorem \ref{thm:ideal-free}). 
We will see in Corollary \ref{cor:BST-IA} and Theorem \ref{cor:BST-NSS} new examples from toric arrangements: The poset of layers of the toric arrangement of a type $B_\ell$ root system for $\ell \ge 3$ is inductive, but not supersolvable. That arises from type $B_2$ depicted in Figure \ref{fig:B2} below is inductive and supersolvable, but not strictly supersolvable. 

Thus for locally geometric posets, we have proved the following:
$$ \tbf{SSS}  \subsetneq   \tbf{IP} \subsetneq \tbf{DP}   \subsetneq \tbf{FR}.$$

Compared with the relation described in Remark \ref{rem:all-property}, supersolvable posets do not form a subclass of inductive posets. 
The poset of layers of the toric arrangement of a type $D_2$ root system (the subposet of the poset in Figure \ref{fig:B2} generated by $\{t_1t_2=1, t_1t_2^{-1}=1\}$) is supersolvable but not inductive. 

The containment $ \tbf{IP} \subsetneq \tbf{DP}$ is strict by an example from Remark \ref{rem:all-property}. 
It remains unknown to us whether or not there exists a divisional but not inductive poset among non-lattice, locally geometric posets. 
\end{remark}

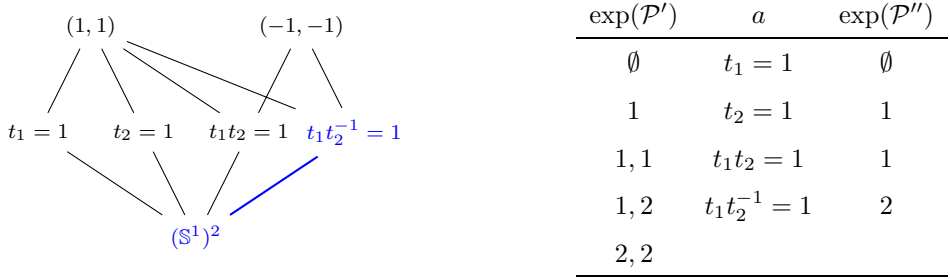
\begin{figure}[htbp]
\centering
\begin{subfigure}{.35\textwidth}
  \centering
\begin{tikzpicture}[scale=.7]
  \node[blue] (0) at (0,0) {\tiny $( \mathbb{S}^1)^2$};
  \node (2) at (-3,2) {\tiny $t_1=1$};
    \node (3) at (-1,2) {\tiny $t_2=1$};
      \node (4) at (1,2) {\tiny $t_1t_2=1$};
        \node[blue] (5) at  (3,2)  {\tiny $t_1t_2^{-1}=1$};
          \node (8) at  (-2,4) {\tiny $(1,1)$};
            \node (9) at  (2,4)  {\tiny $(-1,-1)$};
\draw (0)--(2)--(8)--(3)--(0)--(4)--(8)--(5)--(9)--(4);
\draw[thick, color=blue] (0)--(5);
\end{tikzpicture}
\label{fig:B2-poset}
\end{subfigure}%
\qquad
\begin{subfigure}{.45\textwidth}
  \centering
  {\footnotesize\renewcommand\arraystretch{1.5} 
\begin{tabular}{ccc}
\hline
$\exp(\Pc')$ &  $a$ &   $\exp(\Pc'')$  \\
\hline
$\emptyset$ &  $t_1=1$ & $\emptyset$ \\
$1$ &  $t_2=1$ & $1$ \\
$1,1$ &  $t_1t_2=1$ & $1$ \\
$1,2$ &  $t_1t_2^{-1}=1$ & $2$ \\
$2,2$ &    &   \\
\hline
\end{tabular}
}
  \label{fig:B2-table}
\end{subfigure}
\caption{The toric arrangement of a type $B_2$ root system with its poset $\Pc$ of layers (left) and an induction table for inductiveness (right). 
The  induction table is derived thanks to Theorem \ref{thm:add-IP} which deduces that $\Pc$ is inductive with exponents $\exp(\Pc) =  \{2,2\}$.
In addition, $\Pc$ is supersolvable with the elements of a rank-$1$ M-ideal colored in {\color{blue}{blue}}. 
However, $\Pc$ is not strictly supersolvable since it has no TM-ideal of rank $1$.
}
\label{fig:B2}
\end{figure}

\section{Inductive and divisional abelian arrangements} 
\label{sec:ind-abel-arr}
We first recall preliminary concepts and results of \emph{abelian Lie group} arrangements, or abelian arrangements for short, following  \cite{TY19, LTY21, Bib22}.

Let $G$ be a finite-dimensional connected abelian Lie group, i.e., $G\simeq (\mathbb{S}^1)^a\times \R^b$ for some nonnegative integers $a,b \ge 0$. 
Denote $g:= \dim_\R(G) = a+b$.
Let $\Gamma \simeq \Z^\ell$ be a finite-rank free abelian group. 
We regard $T = \Hom(\Gamma, G) \simeq G^\ell$ with $\dim_\R(T) = g\ell$ as our ambient group. 
For $\alpha \in \Gamma \setminus\{0\}$ and $c \in G$, the \emph{abelian hyperplane} $H_{\alpha,c} :=H_{\alpha,c, G}$ associated to the pair $(\alpha,c)$ is defined by
\begin{equation*}
H_{\alpha, c}:=\{\varphi \in T\mid \varphi (\alpha)= c\} . 
\end{equation*}

Let $\A := \{(\alpha_1,c_1),\ldots, (\alpha_n,c_n) \} \subseteq (\Gamma \setminus\{0\}) \times G$ be a finite set.
We define the \emph{abelian arrangement} $\Ac:=\Ac(\A, G)$ as the collection of connected components of the abelian hyperplanes defined by $\A$ 
$$\Ac:=\{\mbox{connected components of } H_{\alpha, c}\mid (\alpha, c) \in \A\}.$$
We continue to use the notation $\varnothing_\ell$ to denote the empty abelian arrangement in $T \simeq G^\ell$. 
The arrangement $\Ac$ is called \emph{central} if $c_i = 0_G$ for all  $1 \le i \le n$.

When $G = \R^b$ and $\Gamma = \Z^\ell$, we obtain $\Ac$ as an arrangement of affine subspaces in $T \simeq \R^{b\ell}$, and in particular a real (or complex) affine hyperplane arrangement when $b = 1$ ($b = 2$, resp.). 
We sometimes call these hyperplane arrangements \emph{integral}  arrangements as the coefficients of the defining equation of any hyperplane are integer. 
When $G =  \mathbb{S}^1$ (or $G= \mathbb{C}^\times \simeq \mathbb{S}^1\times \R$) and $\Gamma = \Z^\ell$, we obtain an arrangement of real (complex, resp.) translated hypertori or toric arrangement.

For each $\Bc \subseteq \Ac$, denote
$$
H_{\Bc}:=\bigcap_{H\in\Bc}H.
$$
We agree that $H_{\varnothing}:= T$.

The \emph{intersection poset} $L:=L(\Ac)$ of $\Ac$ is defined by
\begin{equation*}
L:=\{\mbox{connected components of nonempty $H_{\Bc}$} \mid \Bc \subseteq \Ac\},
\end{equation*}
whose elements, called \emph{layers}, are ordered by reverse inclusion ($X \le_L Y$ if $X \supseteq Y$). 
Thus $L$ is a pure, ranked poset with a rank function $\rk(X) = \codim(X)/g$ for every $X \in L$. 
The minimal element of $L$ is $\hat 0 = T$, and the atoms of $L$ are the elements of $\Ac$.
 
 \begin{definition} 
\label{def:comb-prop-abe}
Similar to the case of a hyperplane arrangement in an arbitrary vector space, we also refer to the poset $L$ of layers as the \emph{combinatorics} of the abelian arrangement $\Ac$. 
Likewise, a \emph{combinatorial property} of abelian arrangements is defined analogously to Definition \ref{def:comb-prop}.
\end{definition}

Define $\rk(\Ac)$ to be the rank of $L$, i.e., the rank of a maximal element in $L$. 
The arrangement $\Ac$ is called \emph{essential} if $\rk(\Ac)=\ell$.

   \begin{theorem}[{\cite[Corollary 13.11]{Bib22}, \cite[Corollary 4.4.6]{BD22}}]
\label{thm:GP}
Let $\Ac$ be an abelian arrangement. Then $L(\Ac)$ is a geometric poset.
\end{theorem}

The \emph{characteristic polynomial} $\chi_{\Ac}(t)$ of $\Ac$ is defined by 
\begin{equation*}
\chi_{\Ac}(t):= \sum_{X \in L}\mu(T, X)t^{\dim_\R(X)}.
\end{equation*}
Here $\mu:=\mu_L$ is the M\"{o}bius function of $L$. 

\begin{remark} 
\label{rem:chapol-AnL}
 Note that $\chi_{\Ac}(t) =t^{g(\ell - \rk(\Ac))} \cdot \chi_{L}(t^g)$ which has degree $g\ell$. 
 In particular, if $\Ac$ is essential and $g=1$, then $\chi_{\Ac}(t) = \chi_{L}(t)$.
\end{remark}

\begin{definition} 
\label{def:SS-abearr}
Similar to Definition \ref{def:SS-arr}, we call an abelian arrangement $\Ac$   \emph{supersolvable} (resp.,  \emph{strictly supersolvable)} if its intersection poset $L(\Ac)$ is supersolvable (resp., strictly supersolvable).
Denote also by $\mathbf{SS}$ and $\mathbf{SSS}$ the classes of supersolvable and strictly supersolvable abelian arrangements, respectively.
\end{definition}

\begin{definition} 
\label{def:abearr-exp}
Similar to Definition \ref{def:hyparr-exp}, 
we call an abelian arrangement $\Ac$ \emph{factorable} if its intersection poset $L(\Ac)$ is factorable.
In this case, we call the roots of $\chi_{\Ac}(t^{1/g})$ the \emph{(combinatorial) exponents} of $\Ac$ and use the notation $\exp(\Ac)$ to denote the multiset of exponents. 
Denote also by $\mathbf{FR}$ the class of factorable abelian arrangements.
\end{definition} 

By Remark \ref{rem:chapol-AnL}, $\Ac \in \mathbf{FR}$ if and only if there are positive integers $d_1,\ldots,d_{ \rk(\Ac)} \in \Z_{>0}$  such that 
$$\chi_{\Ac}(t) =t^{g(\ell - \rk(\Ac))} \cdot  \prod_{i=1}^{ \rk(\Ac)}(t^g-d_i).$$
 In this case, 
$$ \exp(\Ac) = \{0^{\ell - \rk(\Ac)} \} \cup \exp(L(\Ac)) .$$

\begin{definition}[{\cite[Definitions 13.5 and 13.7]{Bib22}}] 
\label{def:loc-res}
  For each $X \in L$, define 
  $$\A_X : = \{ \alpha \in \Gamma \mid (\alpha, c) \in \A \mbox{ and } H_{\alpha, c} \supseteq X \mbox{ for some } c \in G\}.$$
  The \emph{localization} $\Ac_X$ of $\Ac$ at $X$ is defined as the collection of linear subspaces $H_{\alpha, 0} \subseteq \Hom(\Gamma, \R^g)$ with $\alpha \in \A_X$.
  
For   $H\in\Ac$, the \emph{restriction} $\Ac^H$ of $\Ac$ to $H$ is defined by
$$\Ac^H := \{ \mbox{connected components of nonempty } K \cap H \mid K \in\Ac \setminus \{H\} \}.$$ 
Thus $\Ac^H$ is an arrangement in $H \simeq G^{\ell-1}$.
\end{definition}

The following is well-known, e.g., used in the proof of \cite[Theorem 13.10]{Bib22}. 
 \begin{lemma}
\label{lem:loc-res}
 Let $\Ac$ be an abelian arrangement.  
 Let $X \in L(\Ac)$ and $H\in\Ac$. Then $L(\Ac_X) \simeq L(\Ac)_{\le X}$ and $L(\Ac^H) = L(\Ac)_{\ge H}$.
 \end{lemma}

Fix $H\in\Ac$, define the \emph{deletion} $\Ac':=\Ac \setminus \{H\}$ as an arrangement in $T$, and  $\Ac'' := \Ac^H.$
We call $(\Ac, \Ac', \Ac'')$ the triple of arrangements associated to $H$.
From Definition \ref{def:triple-posets} and Lemma \ref{lem:loc-res}, we have
that $ L(\Ac') = L'$ and $ L(\Ac'') = L''$. 
   \begin{theorem}
\label{thm:DC-formula-A(G)}
Let $\Ac$ be a nonempty abelian arrangement and $H\in\Ac$.  
The following deletion-restriction formula holds
\begin{equation*}
\chi_{\Ac}(t)=
\chi_{\Ac'}(t)-
\chi_{\Ac''}(t). 
\end{equation*}
\end{theorem}

  \begin{proof}
Apply Theorems \ref{thm:LGP-DC}, \ref{thm:GP} and Remark \ref{rem:chapol-AnL}.
\end{proof}

We are ready to introduce the concepts of inductive and divisional abelian arrangements. 

\begin{definition} 
\label{def:IA}
The class $\mathbf{IA}$ of \emph{inductive (abelian) arrangements} is the smallest class of abelian arrangements which satisfies
\begin{enumerate}[(1)]
\item $\varnothing_\ell\in\mathbf{IA}$ for $\ell \ge 1$,
\item $\Ac \in \mathbf{IA}$ if there exists $H \in \Ac$ such that $\Ac'' \in \mathbf{IA}$, $\Ac' \in \mathbf{IA}$, and $\chi_{\Ac'}(t) = (t^g-d)\cdot\chi_{\Ac''}(t) $ for some $d \in \Z$.
\end{enumerate}
\end{definition}

\begin{definition} 
\label{def:DA}
The class $\mathbf{DA}$ of \emph{divisional (abelian) arrangements} is the smallest class of abelian arrangements which satisfies
\begin{enumerate}[(1)]
\item $\varnothing_\ell\in\mathbf{DA}$ for $\ell \ge 1$,
\item $\Ac \in \mathbf{DA}$ if there exists $H \in \Ac$ such that $\Ac'' \in \mathbf{DA}$ and $\chi_{\Ac}(t) = (t^g-d)\cdot\chi_{\Ac''}(t) $ for some $d \in \Z$.
\end{enumerate}
\end{definition}

We now show that inductiveness and divisionality depend only on the combinatorics of arrangements.

   \begin{theorem}
\label{thm:IP=IA-DP=DA}
Let $\Ac$ be an abelian arrangement. 
Then $\Ac \in \mathbf{IA}$ (resp., $\mathbf{DA}$) if and only if $L(\Ac) \in \mathbf{IP}$ (resp., $\mathbf{DP}$). 
\end{theorem}

  \begin{proof}
  We show the assertion for inductiveness by double induction on $\rk(\Ac)$ and $|\Ac|$. 
 The assertion for divisionality can be proved by induction on $\rk(\Ac)$ by a similar (and easier) argument. 

The assertion is clearly true when $\rk(\Ac)=0$ or $|\Ac|=0$ (i.e., $\Ac =\varnothing$). 
Suppose $\rk(\Ac) \ge 1$ and $|\Ac|\ge 1$.
Suppose $\Ac \in \mathbf{IA}$. 
Then there exists $H \in \Ac$ such that $\Ac'' \in \mathbf{IA}$, $\Ac' \in \mathbf{IA}$, and $\chi_{\Ac'}(t) = (t^g-d)\cdot\chi_{\Ac''}(t) $ for some $d \in \Z$.
Note that $|\Ac'| < |\Ac|$ and $\rk(\Ac'') < \rk(\Ac)$. 
By the induction hypothesis, $L'' = L(\Ac'') \in \mathbf{IP}$ and $L' = L(\Ac') \in \mathbf{IP}$. 
Moreover, if $\rk(\Ac) = \rk(\Ac') + 1$, then by Remark \ref{rem:chapol-AnL}, 
$$t^g \cdot \chi_{L'}(t^g) = (t^g -d)\cdot \chi_{L''}(t^g).$$
Hence $\chi_{L'}(t) = \chi_{L''}(t)$ since $t \nmid  \chi_{L''}(t)$. 
Similarly, if $\rk(\Ac) = \rk(\Ac')$, then $\chi_{L'}(t) =(t-d) \chi_{L''}(t)$. 
In either case, $\chi_{L''}(t)$ divides $\chi_{L'}(t)$. 
Thus $L(\Ac) \in \mathbf{IP}$. 
A similar argument shows that if $L \in \mathbf{IP}$ then $\Ac \in \mathbf{IA}$, which completes the proof.
\end{proof}

   \begin{corollary}
\label{cor:IA-DA-combinatorial}
The property of being inductive or divisional of an abelian arrangement is a combinatorial property. 
\end{corollary}

  \begin{proof}
 It follows from Proposition \ref{prop:IP-DP-preserved} and Theorem \ref{thm:IP=IA-DP=DA} above.
\end{proof}

\begin{remark} 
\label{rem:all-property-abe}
By Remark \ref{rem:JT-SSS<IP} and Theorem \ref{thm:IP=IA-DP=DA}, we have the following:
$$ \tbf{SSS}  \subsetneq   \tbf{IA} \subseteq \tbf{DA} \subsetneq \tbf{FR}.$$

It is an open question to us whether or not the containment $ \tbf{IA} \subseteq \tbf{DA}$ is strict. 
This is related to the question in the last paragraph in Remark \ref{rem:JT-SSS<IP}.
The example of a hyperplane arrangement that is divisionally free but not inductively free in Remark \ref{rem:all-property} is not an integral arrangement. 
\end{remark}

An abelian arrangement is inductive if it can be constructed from the empty arrangement by adding an element ($=$\,a connected component of a hyperplane) one at a time with the aid of  the following ``addition" theorem at each addition step. 
It thus also makes sense to speak of an induction table for an inductive arrangement in a similar way as of inductive posets in Section \ref{sec:ind-poset}. 
   \begin{theorem}
\label{thm:IA-add}
Let $\Ac \ne \varnothing$ be an abelian arrangement   in $T\simeq G^\ell$ and let $H \in \Ac$. 
If $\Ac'' \in \mathbf{IA}$ with $\exp(\Ac'') = \{d_1,\ldots,d_{\ell-1}\}$ and $\Ac' \in \mathbf{IA}$ with $\exp(\Ac') = \{d_1,\ldots,d_{\ell-1},d_\ell\}$, then $\Ac \in \mathbf{IA}$ with $\exp(\Ac) = \{d_1,\ldots,d_{\ell-1},d_\ell+1\}$.
\end{theorem}

  \begin{proof}
It follows directly from Definition \ref{def:IA} and Theorem \ref{thm:DC-formula-A(G)}.
\end{proof}

We complete this section by describing an  arrangement theoretic characterization for (strict) supersolvability.

 \begin{definition}
\label{def:M-TM-ideal-arr}
Given a subarrangement $\Bc$ of an abelian arrangement $\Ac$, we say $\Bc$ is an \emph{M-ideal} of $\Ac$ if $L(\Bc)$ is a proper order ideal of $L(\Ac)$, and for any two distinct $H_1,H_2 \in \Ac \setminus \Bc$ and every connected component $C$ of the intersection $H_1\cap H_2$  there exists $H_3 \in  \Bc$ such that   $C \subseteq H_3$.
More strongly, an M-ideal $\Bc$ is called a \emph{TM-ideal} of $\Ac$ if 
\begin{enumerate}
    \item[(*)] for any $X \in L(\Bc)$ and $H  \in \Ac \setminus \Bc$ the intersection $X\cap H$ is connected. 
\end{enumerate}

\end{definition}

   \begin{theorem}
\label{thm:SS=SS}
Let $\Ac$ be an arrangement of rank $r$ in $T\simeq G^\ell$. 
Then $\Ac$ is supersolvable  (resp., strictly supersolvable)  (Definition \ref{def:SS-abearr}) if and only if there is a chain, called  an \emph{M-chain} (resp., a \emph{TM-chain)}
$$\varnothing= \Ac_0 \subseteq \Ac_1 \subseteq \cdots \subseteq  \Ac_r =\Ac,$$
such that each $\Ac_i$ is an M-ideal  (resp., a TM-ideal)  of $\Ac_{i+1}$.
\end{theorem}

  \begin{proof} 
Observe that if $\Bc\subseteq \Ac$, then $L(\Bc)$ is a pure, join-closed ideal of $L(\Ac)$.  
Note also that the poset of layers of an abelian arrangement is a geometric poset by Theorem \ref{thm:GP}. 
Thus by Lemma \ref{lem:GPMI}, if $\Bc$ is an M-ideal (resp., a TM-ideal) of $\Ac$, then $L(\Bc)$ is an M-ideal (resp., a TM-ideal) of $L(\Ac)$ with $\rk(\Bc) = \rk(\Ac)-1$. 
Therefore, if there exists an M-chain (resp., a TM-chain)
$$\varnothing= \Ac_0 \subseteq \Ac_1 \subseteq \cdots \subseteq  \Ac_r =\Ac,$$
then $L(\Ac)$ is supersolvable  (resp., strictly supersolvable)  with an M-chain (resp., a TM-chain)
$$\{\hat 0\} = L(\varnothing)  \subseteq L(\Ac_1) \subseteq \cdots \subseteq  L(\Ac_r) =L(\Ac),$$

Conversely, if $\Qc$ is an M-ideal (resp., a TM-ideal)  of $L(\Ac)$ with $\rk(\Qc) = \rk(\Ac)-1$, then again by  Lemma \ref{lem:GPMI}, the set $A(\Qc)$ of atoms is an M-ideal (resp., a TM-ideal) of $\Ac$. 
Thus if $L(\Ac)$ is supersolvable  (resp., strictly supersolvable), then any M-chain (resp., TM-chain) of $L(\Ac)$ induces  an M-chain (resp., a TM-chain) for $\Ac$.
\end{proof}

\section{Localization of hyperplane and toric arrangements}
\label{sec:local}

In this section, we discuss the operation of localizing at a layer of an abelian arrangement in the sense of Definition \ref{def:loc-res}. 
Note from Remark \ref{rem:S} that (strict) supersolvability is closed under taking localization:  If $\Ac \in \mathbf{SS}$ (resp., $\Ac \in \mathbf{SSS}$), then $\Ac_X \in \mathbf{SS}$ (resp., $\Ac_X \in \mathbf{SSS}$)   for every $X \in L(\Ac)$. 
We will see that in general it is not the case for inductiveness or divisionality.
More explicitly, we give an example of an inductive toric arrangement with a non-factorable localization.

First let us recall from the previous section the definition of central (real) hyperplane and toric arrangements as abelian arrangements when the Lie group $G$ is $\R$ and $\mathbb{S}^1$, respectively. 
Let $\A$ be a finite set of integral vectors in $\Z^\ell$.
Given a vector $\alpha=(a_1,\ldots,a_\ell) \in \A$, we may define the hyperplane 
\begin{equation*}
H_{\alpha, \R}:=\{x \in \R^\ell \mid a_1x_1 + \cdots + a_\ell x_\ell= 0\},
\end{equation*}
and the hypertorus
\begin{equation*}
H_{\alpha, \mathbb{S}^1}:=\{t \in (\mathbb{S}^1)^\ell \mid t_1^{a_1} \cdots  t_\ell^{a_\ell}= 1\}.
\end{equation*}

The set $\A \subseteq \Z^\ell$ defines the central hyperplane arrangement
$$\Hc:=\{ H_{\alpha, \R}\mid \alpha \in \A\}.$$
and the central toric arrangement
$$\Ac:=\{\mbox{connected components of } H_{\alpha, \mathbb{S}^1} \mid \alpha \in \A\}.$$

Alternatively, given an integral matrix $S \in \Mat_{\ell \times m}(\Z)$, we may view each column as a vector in $\Z^\ell$ so that we may define the central hyperplane and toric arrangements from $S$ as above.

\begin{example}\label{eg:diShi321}   
Let $S \in \Mat_{3 \times 6}(\Z)$ be an integral matrix defined as below:
\begin{equation}
\label{eq:matrixS}
S =\begin{bmatrix}
1 & 0  & 1  & 0 & 1  & 0  \\
0 & 1 &  1 & 0 & 0 & 1\\
0 & 0 &  0 & 1 & -1 & -1
\end{bmatrix}.
\end{equation}

Let $\Hc_S$ and $\Ac_S$ be the central hyperplane and toric arrangements defined by $S$, respectively. 
Note that by definition of localization (Definition \ref{def:loc-res}) we may write $\Hc_S = (\Ac_S)_X$ where $X$ denotes the layer $(1,1,1) \in L(\Ac_S)$.

In fact, $\Hc_S$ is linearly isomorphic to the \emph{essentialization} of the \emph{cone} of the \emph{digraphic Shi arrangement} defined by the path $3 \to 2 \to 1$ in \cite[Figure 3]{Atha98}. 
The characteristic polynomial of $\Hc_S$ is given by
$$\chi_{\Hc_S}(t) = (t-1)(t^2-5t+7),$$
which implies that $\Hc_S$ is not divisional hence not inductive.

However, we may show that $\Ac_S$ is inductive with exponents $\{2,2,2\}$. 
Let $H_i$ denote the (connected) hypertorus defined by the $i$-th column of the matrix $S$. 
The poset of layers of $\Ac_S$ and an induction table are given in Figure \ref{fig:toric-inductive}. 
(Observe also that $\Ac_S$ is not locally supersolvable since the localization $\Hc_S$ is not supersolvable by the preceding discussion.)
\end{example}

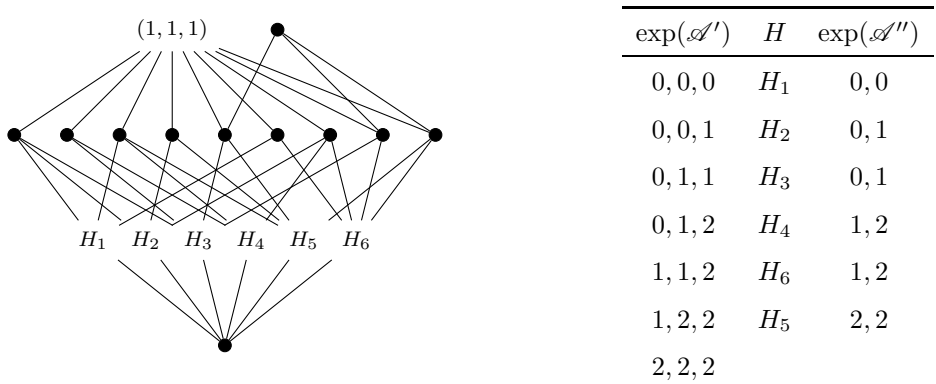
\begin{figure}[htbp]
\centering
\begin{subfigure}{.35\textwidth}
  \centering
\begin{tikzpicture}[scale=.7]

  \node (1) at (-2.5,2) {\tiny $H_1$};
    \node (2) at (-1.5,2) {\tiny $H_2$};
    \node (3) at (-.5,2) {\tiny $H_3$};
      \node (4) at (.5,2) {\tiny $H_4$};
        \node (5) at  (1.5,2)  {\tiny $H_5$};
                \node (6) at  (2.5,2) {\tiny $H_6$};              
                                \node (16) at  (-1,6) {\tiny $(1,1,1)$};

\draw (0,0) node[v](0){};

\draw (-4,4) node[v](7){};
\draw (-3,4) node[v](8){};
\draw (-2,4) node[v](9){};
\draw (-1,4) node[v](10){};
\draw (0,4) node[v](11){};
\draw (1,4) node[v](12){};
\draw (2,4) node[v](13){};
\draw (3,4) node[v](14){};
\draw (4,4) node[v](15){};

\draw (1,6) node[v](17){};

\draw (0)--(1);
\draw (0)--(2);
\draw (0)--(3);
\draw (0)--(4);
\draw (0)--(5);
\draw (0)--(6);

 \draw (1)--(7);
 \draw (1)--(9);
 \draw (1)--(12);
 
  \draw (2)--(7);
 \draw (2)--(10);
 \draw (2)--(13);
 
   \draw (3)--(7);
 \draw (3)--(8);
 \draw (3)--(11);
   \draw (3)--(14);

  \draw (4)--(8);
 \draw (4)--(9);
 \draw (4)--(13); 
 
    \draw (5)--(9);
 \draw (5)--(10);
 \draw (5)--(11);
   \draw (5)--(15);
 
    \draw (6)--(12);
 \draw (6)--(13);
 \draw (6)--(14);
   \draw (6)--(15);
   
   \draw (16)--(7);
   \draw (16)--(8);
   \draw (16)--(9);
   \draw (16)--(10);
   \draw (16)--(11);
   \draw (16)--(12);
      \draw (16)--(13);
         \draw (16)--(14);
            \draw (16)--(15);
            
                  \draw (17)--(11);
         \draw (17)--(14);
            \draw (17)--(15);
 
\end{tikzpicture}
\label{fig:222-poset}
\end{subfigure}%
\qquad
\begin{subfigure}{.45\textwidth}
  \centering
{\footnotesize\renewcommand\arraystretch{1.5} 
\begin{tabular}{ccc}
\hline
$\exp(\Ac')$ &  $H$ &   $\exp(\Ac'')$  \\
\hline
$0,0,0$ &  $H_1$ & $0,0$ \\
$0,0,1$ &  $H_2$ & $0,1$ \\
$0,1,1$ &  $H_3$ & $0,1$ \\
$0,1,2$ &  $H_4$ & $1,2$ \\
$1,1,2$ &  $H_6$ & $1,2$ \\
$1,2,2$ &  $H_5$ & $2,2$ \\
$2,2,2$ & & \\
\hline
\end{tabular}
}
  \label{tab:123}
\end{subfigure}
\caption{The poset of layers of the toric arrangement $\Ac_S$ defined by matrix $S$ in \eqref{eq:matrixS} and an induction table for its inductiveness.
}
\label{fig:toric-inductive}
\end{figure}

It happens quite often that the hyperplane arrangement defined by a matrix is inductive, but the toric arrangement defined by the same matrix is not (see the next section).
Example \ref{eg:diShi321} above deduces that the converse is also possible. 
This is a rare, perhaps counter-intuitive example that toric arrangement could be inductive, while hyperplane arrangement cannot be. 

\section{Application to toric arrangements of ideals of root systems}
\label{sec:rootsystem}
Our standard reference for root systems is \cite{B68}.
Let $\Phi$ be an irreducible (crystallographic) root system in $V=\R^\ell$. 
Fix a positive system $\Phi^+ \subseteq \Phi$ and the associated set of simple roots (base) $\Delta := \{\alpha_1,\ldots,\alpha_\ell \} \subseteq\Phi^+$. 

Define the partial order $\ge$ on $\Phi^+$ such that $\beta_1 \ge\beta_2$ if and only if $\beta_1-\beta_2 = \sum_{i=1}^\ell n_i\alpha_i$ with all $n_i \in \Z_{\ge 0}$. 
A subset $\I\subseteq\Phi^+$ is called an \emph{ideal} if, for $\beta_1,\beta_2 \in \Phi^+$, $\beta_1 \ge \beta_2, \beta_ 1 \in \I$ then $\beta_2 \in \I$. 

For $\beta = \sum_{i=1}^\ell n_i\alpha_i \in \Phi^+$, the \emph{height}  of $\beta$ is defined by $ {\rm ht}(\beta) := \sum_{i=1}^\ell n_i $. 
Let $\I$ be an ideal of  $ \Phi^+$ and set $M:=\max\{{\rm ht}(\beta)\mid \beta \in \I\}$. 
Let $t_k: = |\{ \beta \in \I \mid  {\rm ht}(\beta)=k \}|$ for $1 \le k \le M$. 
The sequence $(t_1, \ldots , t_k, \ldots , t_{M})$ is called the \textit{height distribution} of $\I$. 
The  \textit{dual partition} $\mathrm{DP}(\I)$ of the height distribution of $\I$ is defined as the multiset of nonnegative integers
$$\mathrm{DP}(\I) := \{0^{\ell-t_1},1^{t_1-t_2},\ldots , M^{t_{M}}\},$$ 

For each $\Psi\subseteq\Phi^+$, let $S_\Psi$ denote the coefficient matrix of $\Psi$ with respect to the base $\Delta$, i.e., $S_\Psi = [s_{ij}]$ is the $\ell \times |\Psi |$ integral matrix that satisfies 
$$\Psi = \left\{\sum_{i=1}^\ell s_{ij}\alpha_i \mid 1 \le j \le |\Psi | \right\}.$$

Note that the matrix $S_\Psi$ depends only upon $\Phi$.

 \begin{definition}
\label{def:root-lattice}
Following the previous section, we define  $\Ac_{\Psi } := \Ac_{S_\Psi }(\Phi) $ and  $\Hc_{\Psi } := \Hc_{S_\Psi }(\Phi) $ as the central toric and hyperplane arrangements defined by $S_\Psi $ respectively. 
We call these arrangements the arrangements with respect to the \emph{root lattice}.
\end{definition}

   \begin{theorem}[\cite{ST06, ABCHT16, Hul16, Roh17, CRS19}]
\label{thm:ideal-free}
If $\I$ is an ideal of an irreducible root system $\Phi$, then $ \Hc_{\I } $ is inductive with exponents $\mathrm{DP}(\I)$. 
Moreover, $  \Hc_{\I }  $ is supersolvable if $\Phi$ is $A_\ell$, $B_\ell$, $C_\ell$, or $G_2$.
\end{theorem}

In contrast to the hyperplane arrangement case, the toric arrangement $\Ac_{\I}$ is not factorable for most cases even when $\I =\Phi^+$.
It is known that  the characteristic polynomial of the central toric arrangement defined by an arbitrary matrix $S$ coincides with the last constituent of the \emph{characteristic quasi-polynomial} $\chi^{\rm \quasi}_S(q)$ defined by $S$ \cite[Corollary 5.6]{LTY21}.
Furthermore, an explicit computation shows that the last constituent of $\chi^{\rm \quasi}_{S_{\Phi^+} }(q)$  factors with all integer roots if and only if $\Phi$ is $A_\ell$, $B_\ell$ or $C_\ell$ \cite{KTT10, Su98}. 
Thus, $\Ac_{{\Phi^+} }$ is factorable if and only if $\Phi$ is of one of these three types.

Even more is true: If $\I$  is an ideal of an irreducible root system of type $A$, $B$ or $C$, then $\Ac_{{\I} }$ is factorable whose combinatorial exponents can be described by the \emph{signed graph} associated to $\I$ \cite{T19}. 
Our third main result Theorem \ref{thm:ABC-IA-intro} strengthens this result.  
Furthermore, we give an explicit description of the exponents of $ \Ac_{\I } $ derived from an explicit induction table. 
This description turns out to be equivalent to the ones in \cite{T19}.
We also give a characterization for supersolvability of $\Ac_{{\Phi^+} }$ when $\Phi$ is of type $B$ (Theorem \ref{cor:BST-NSS}).

\begin{proof}[\tbf{Proof of  Theorem \ref{thm:ABC-IA-intro}}]
It follows from Corollary \ref{cor:ideal-SS-A},  Theorem \ref{thm:CS-IA} and Corollary \ref{cor:BST-IA} below.
\end{proof}

The proof for the type $A$ case in Theorem \ref{thm:ABC-IA-intro} is a simple consequence of Theorem \ref{thm:ideal-free}, which we give below.

   \begin{corollary}
\label{cor:ideal-SS-A}
If $\I$ is an ideal of a root system of type $A$, then the toric arrangement $ \Ac_{\I } $ with respect to the root lattice is strictly supersolvable (equivalently, supersolvable) hence inductive with exponents $\mathrm{DP}(\I)$. 
\end{corollary}

  \begin{proof}
It is not hard to see that for any $\Psi\subseteq\Phi^+(A_\ell)$, each layer in $L(\Ac_{\Psi }(A_\ell))$ is connected. 
Thus $L(\Ac_{\Psi }(A_\ell)) \simeq L(\Hc_{\Psi }(A_\ell))$ which is a geometric lattice. 
By Remark \ref{rem:S}, its supersolvability and strict supersolvability are equivalent.
Moreover, $ \Ac_{\I } $ is indeed supersolvable with exponents $\mathrm{DP}(\I)$ by Theorem \ref{thm:ideal-free}. 
\end{proof}

Hence we are left with the computation on types $B$ and $C$.
First we need a construction of root systems of these types via a choice of basis for $V$ following \cite[Chapter VI, \S4]{B68}.

Let $\E : =\{\epsilon_1, \ldots, \epsilon_{\ell}\}$ be an orthonormal basis for $V$. 
For $\ell \ge 1$,
$$\Phi(B_\ell)  = \{\pm\epsilon_i \,(1 \le i  \le \ell),\pm(\epsilon_i \pm \epsilon_j)  \,(1 \le i < j \le \ell )\}$$
 is an irreducible root system of type $B_{\ell}$. 
We may choose a positive system
$$\Phi^+(B_\ell)  = \{\epsilon_i \,(1 \le i  \le \ell), \epsilon_i \pm \epsilon_j \,(1 \le i < j \le \ell )\}.$$
Define $\alpha_i := \epsilon_i - \epsilon_{i+1}$ for $1 \le i  \le \ell-1$, and $\alpha_{\ell} := \epsilon_{\ell}$. 
Then $\Delta(B_\ell)  = \{\alpha_1,\ldots, \alpha_{\ell}\}$ is the base associated to $\Phi^+(B_\ell) $. 
We may express
\begin{align*}
\Phi^+(B_\ell)  & =  \Bigl\{ \epsilon_i=\sum_{i\le k \le \ell }\alpha_k \,(1 \le i \le \ell), \epsilon_i-\epsilon_j=\sum_{i\le k < j }\alpha_k \,(1 \le i <  j \le \ell), \\
& \epsilon_i+\epsilon_j=\sum_{i\le k<j }\alpha_k+2 \sum_{j\le k \le \ell }\alpha_k\,(1 \le i < j \le \ell) \Bigr\}.
\end{align*}

For $\Psi\subseteq\Phi^+(B_\ell)$, write $T_\Psi = [t_{ij}]$ for the coefficient matrix of $\Psi$ with respect to the basis $\E$. 
The matrices $T_\Psi$ and $S_\Psi$ are related by $T_\Psi=P(B_\ell)\cdot S_\Psi$, where $P(B_\ell) $ is an unimodular matrix of size $\ell \times \ell$ given by
$$P(B_\ell) =\begin{bmatrix}
1 &   &   & &  \\
-1 & 1 &   &  & \\
  & -1 &   &  & \\
  &   & \ddots  &  & \\
    &   &  & 1 & \\
      &   &  & -1 & 1
\end{bmatrix}.$$

Similarly, an irreducible root system of type $C_{\ell}$ for $\ell \ge 1$  is given by
\begin{align*}
\Phi(C_\ell)  &= \{\pm2\epsilon_i \,(1 \le i  \le \ell),\pm(\epsilon_i \pm \epsilon_j)  \,(1 \le i < j \le \ell )\}, \\
\Phi^+(C_\ell)  &= \{2\epsilon_i \,(1 \le i  \le \ell), \epsilon_i \pm \epsilon_j \,(1 \le i < j \le \ell )\}, \\
\Delta(C_\ell) &=\{\alpha_i = \epsilon_i - \epsilon_{i+1}\,(1 \le i  \le \ell-1),\, \alpha_{\ell} =2\epsilon_{\ell} \} ,\\
\Phi^+(C_\ell)  &= \{ 2\epsilon_i=2\sum_{i\le k< \ell }\alpha_k+\alpha_\ell \,(1 \le i \le \ell), \epsilon_i-\epsilon_j=\sum_{i\le k<j }\alpha_k \,(1 \le i < j \le \ell),  \\
& \epsilon_i+\epsilon_j=\sum_{i\le k<j }\alpha_k+2 \sum_{j\le k <\ell }\alpha_k+\alpha_\ell\,(1 \le i < j \le \ell) \}.
\end{align*}
$$P(C_\ell) =\begin{bmatrix}
1 &   &   & &  \\
-1 & 1 &   &  & \\
  & -1 &   &  & \\
  &   & \ddots  &  & \\
    &   &  & 1 & \\
      &   &  & -1 & 2
\end{bmatrix}.$$

\begin{example}
\label{ex:B2C2}   
Let $\Phi=B_2$ with $\Phi^+= \{\alpha_1 = \epsilon_1 - \epsilon_2 ,\alpha_2 = \epsilon_2,  \alpha_1+\alpha_2 = \epsilon_1, \alpha_1+2\alpha_2= \epsilon_1 + \epsilon_2\}$ where $\Delta=\{\alpha_1,\alpha_2\}$ and $\E =\{\epsilon_1,  \epsilon_2\}$.
 The coefficient matrices of  $\Phi^+$ w.r.t. $\Delta$ and $\E$ are given by 
$$
S_{\Phi^+} = 
\begin{pmatrix}
1 & 0 & 1 &   1 \\
0 & 1 & 1 &   2 
\end{pmatrix},
\quad
T_{\Phi^+} = 
\begin{pmatrix}
1 & 0 & 1 &   1 \\
-1 & 1 & 0 &   1 
\end{pmatrix}.
$$
 Let $\Phi=C_2$. 
  The coefficient matrix of  $\Phi^+$ w.r.t. $\Delta$ is $S_{\Phi^+}$ above with rows switched (this is not the case when $\ell \ge 3$). 
 The coefficient matrix of  $\Phi^+$ w.r.t. $\E =\{\epsilon_1,  \epsilon_2\}$  is given by 
$$
T_{\Phi^+} = 
\begin{pmatrix}
1 & 0 & 1 &   2 \\
-1 & 2 & 1 &   0 
\end{pmatrix}.
$$
\end{example}

 \begin{definition}
\label{def:integer-lattice}
Let $\Phi = B_\ell$ or $C_\ell$. 
For $\Psi\subseteq\Phi^+$, denote by  $\Ac_{T_\Psi }$ and  $\Hc_{T_\Psi }$ the central toric and hyperplane arrangements defined by the matrix $T_\Psi $, respectively. 
We call these arrangements the arrangements with respect to the \emph{integer lattice}.
\end{definition}

\begin{remark} 
\label{rem:A-const}
Since the matrix $P(B_\ell)$ is unimodular, for every $\Psi\subseteq\Phi^+(B_\ell)$ we have an isomorphism of posets of layers: $L( \Ac_{{\Psi}  }) \simeq L( \Ac_{T_\Psi } )$ (see e.g., \cite[\S5]{PP21}).
However, $\det P(C_\ell) = 2$. In general, $L( \Ac_{{\Psi} }) \not\simeq L( \Ac_{T_\Psi })$  for $\Psi\subseteq\Phi^+(C_\ell)$ (although $L( \Hc_{{\Psi} }) \simeq L( \Hc_{T_\Psi } )$).

A positive system $\Phi^+(A_{\ell-1})$ of an irreducible root system $\Phi$ of type $A_{\ell-1}$ for $\ell \ge 2$ can be defined as the ideal of $\Phi^+(B_\ell)$ (or $\Phi^+(C_\ell)$) generated by $ \epsilon_1-\epsilon_\ell=\sum_{k=1}^{\ell-1}\alpha_k$. 
Thus $L( \Ac_{{\Psi}  }) \simeq L( \Ac_{T_\Psi } )$ for every $\Psi\subseteq\Phi^+(A_{\ell-1})$.
\end{remark}

To describe the exponents of  $\Ac_{\I}$ when $\Phi$ is $B_\ell$ or $C_\ell$, we need information from the signed graph associated to $\I$.

\begin{definition} 
\label{def:SG}
Let $\Phi = B_\ell$ or $C_\ell$. 
For $\Psi\subseteq\Phi^+$ and $1 \le i \le \ell$, define the subset $E_i = E_i(\Psi)  \subseteq \Psi$ by 
$$E_i := E^+_i \sqcup E^-_i, \, \mbox{where}\, 
E^{\pm}_i := \{ \epsilon_i\pm\epsilon_j \in \Psi \mid i <j \} \subseteq \Psi.
$$
For $\alpha \in E_i$, let $H_\alpha$ denote the hypertorus defined by $\alpha$. 
For example, $\alpha =  \epsilon_i+\epsilon_j $ defines the hypertorus $H_\alpha =\{t_i t_j =1\}$.
We then define the subarrangement $\Bc_i = \Bc_i(\Psi)\subseteq \Ac_{{\Psi} }$ by 
$$\Bc_i  :=  \Bc^+_i \sqcup  \Bc^+_i ,\, \mbox{where}\, \Bc^{\pm}_i  := \{H_\alpha \mid \alpha \in E^{\pm}_i  \} \subseteq \Ac_\Psi.
$$
Finally, define $b^{\pm}_i  : = |\Bc^{\pm}_i |$ and $b_i : = |\Bc_i | = b^+_i + b^-_i$.
\end{definition}

In the language of  signed graphs (e.g., following \cite[\S5]{Z81}), the elements in $E^+_i(\Psi)$ and $E^-_i(\Psi)$ correspond to the \emph{positive} and  \emph{negative edges} of the signed graph defined by $\Psi$, respectively.

It is not hard to see that for each ideal $\I$ of $\Phi^+(B_\ell)$ or $\Phi^+(C_\ell)$, the elements of the dual partition $\mathrm{DP}(\I)$ can be expressed in terms of $b_i(\I)$'s and vice versa. 
However, the numbers $b_i$'s are a bit more convenient for our subsequent discussion. 

\subsection{Type \texorpdfstring{$C$}{C}.}
We first present the results on type $C$ as the proofs are simpler than those on type $B$. 
We begin by proving a lemma which serves as a template for some arguments later. 
   \begin{lemma}
\label{lem:CT-TM}
Let $\I \subseteq \Phi^+(C_\ell)$ be an ideal such that $E_1(\I) \ne \emptyset $. 
Define
$$\scD  :=  \begin{cases}
\I \setminus (E_1(\I) \cup \{2\epsilon_1\}) & \mbox{ if $2\epsilon_1 \in \I$}, \\
\I \setminus E_1(\I)  & \mbox{ otherwise}.
 \end{cases} 
 $$
Then $\scD$ can be regarded as an ideal of $\Phi^+(C_{\ell-1})$ and $  \Ac_{T_\scD}$ is a TM-ideal of $\Ac_{T_\I }$. 
\end{lemma}

  \begin{proof}
  The first assertion is clear via the transformation $x_i \mapsto x_{i-1}$ for $2 \le i \le \ell$. 
  Denote $\Ac :=\Ac_{T_\I }$ and $ \Dd : =  \Ac_{T_\scD}$.
There do not exist $X \in L( \Dd)$ and $Y \in  L(\Ac)  \setminus L( \Dd)$ such that $X \subseteq Y$ since  the defining equations of any $X \in L( \Dd)$ do not involve $t_1$. 
Therefore, $L(\Dd)$ is a proper order ideal of $L(\Ac)$. 
Note also that the power of variable $t_1$ in the defining equation of any $H \in\Ac  \setminus \Dd$ is equal to $1$. 
This shows Condition \ref{def:M-TM-ideal-arr}(*).  

It remains to show that for any two distinct $H_1,H_2 \in \Ac  \setminus \Dd$ and every connected component $C$ of the intersection $H_1\cap H_2$,  there exists $H_3 \in  \Dd$ such that   $C \subseteq H_3$. 
We consider three main cases, the remaining cases are similar to one of these.
\begin{enumerate}[(a)]
	\item Assume $H_1 = \{ t_1t_j =1\}$ (i.e., $\epsilon_1+\epsilon_j \in \I$) and  $H_2= \{ t_1t^{-1}_k =1\} $ for $j>1$, $k>1$, $j \ne k$. 
Then by the definition of an ideal we must have $\epsilon_j+\epsilon_k \in \scD$ (since $\epsilon_1+\epsilon_j  > \epsilon_j+\epsilon_k$). 
Hence $H_3 := \{ t_jt_k =1\} \in \Dd$. 
 Moreover, $H_1\cap H_2$ is connected and $H_1\cap H_2 \subseteq H_3$. 
 	\item Assume $H_1 = \{ t_1t_j =1\}$ and  $H_2= \{ t_1t^{-1}_j =1\} $ for $j>1$. 
Then $H_3 := \{ t_j =1\} \in \Dd$ and $H'_3 := \{ t_j =-1\} \in \Dd$ (since $\epsilon_1+\epsilon_j  > 2\epsilon_j$). 
 Moreover, $H_1\cap H_2$ has two connected components; one is contained in $H_3$, the other is contained in $H'_3$.
  	\item Assume $H_1 = \{ t_1=1\}$ (i.e., $2\epsilon_1 \in \I$) and  $H_2= \{ t_1t_j =1\} $ for $j>1$. 
Then $H_3 := \{ t_j =1\} \in \Dd$  (since $2\epsilon_1 > 2\epsilon_j$). 
 Moreover, $H_1\cap H_2$ is connected and $H_1\cap H_2 \subseteq H_3$. 
\end{enumerate}
 
This concludes that $\Dd$ is a TM-ideal of $\Ac$ as desired.
\end{proof}

   \begin{theorem}
\label{thm:CT-SSS}
Let $\I \subseteq \Phi^+(C_\ell)$ be an ideal. Define 
\begin{align*}
n & :=   \begin{cases}
 \min \{ 1 \le i \le \ell \mid E_i(\I) \ne \emptyset \} & \mbox{ if } \I \ne \emptyset, \\
\ell+1 & \mbox{ otherwise},
 \end{cases} \\
s & :=  \begin{cases}
\min \{ 1 \le i \le \ell \mid 2\epsilon_i \in \I\} & \mbox{ if there exists $2\epsilon_i \in \I$ for some $1 \le i \le \ell $}, \\
\ell+1 & \mbox{ otherwise}.
 \end{cases} 
\end{align*}
Then the toric arrangement  $\Ac_{T_\I }$ with respect to the integer lattice is strictly supersolvable with exponents
$$\exp(\Ac_{T_\I }) = \{0^{n-1} \} \cup \{ b_i\}_{i=n}^{s-1} \cup \{ 2(\ell-i+1)\}_{i=s}^{\ell}.$$ 
(See Definition \ref{def:SG} for the definition of $b_i$'s.)
\end{theorem}

  \begin{proof}
  Denote $\Ac :=\Ac_{T_\I }$.
  Note that $n \le s$ and $b_i =0$ for $1 \le i <n$.
  If $2\epsilon_i\notin \I$ for all $1 \le i \le \ell $, then $\I$ can be regarded as an ideal of $\Phi^+(A_{\ell-1})$ by Remark \ref{rem:A-const}. 
  Thus, $L( \Ac_{{\I}  }) \simeq L( \Ac_{T_\I } )$. 
  By Corollary \ref{cor:ideal-SS-A}, $\Ac \in \mathbf{SSS}$ with exponents $\mathrm{DP}(\I) = \{b_1, \ldots, b_{\ell}\}$. 
  
Now we may assume $1 \le n \le s \le \ell$. 
  Then $ 2\epsilon_i \in \I$ and $E_i(\I) \ne \emptyset $ for all $s \le i \le \ell$. 
  Define 
 $$
 \Ac_i :=
 \begin{cases}
 \bigcup_{j=i}^\ell \left( \Bc_j \cup \{t^2_j =1 \}\right) & \mbox{ if } s \le i \le \ell, \\
 \bigcup_{j=i}^{s-1} \Bc_j  \cup \Ac_s & \mbox{ if } n \le i < s.
 \end{cases}
 $$
 In particular, $\Ac_s$ can be identified with $\Ac_{T_{\Phi^+}} (C_{\ell-s+1})$ (via $x_i \mapsto x_{i-s+1}$ for $s \le i \le \ell$). 
 Then $b_i =2(\ell-i)$ for $s \le i \le \ell$.
 
 By Theorem \ref{thm:SS=SS}, it suffices to show that the chain
 $$\varnothing \subsetneq \Ac_\ell \subsetneq \cdots \subsetneq  \Ac_n =\Ac$$
 is a TM-chain of $\Ac$. 
A similar argument as in the proof of Lemma \ref{lem:CT-TM} shows that $\Ac_{i+1} $ is a TM-ideal of $\Ac_{i} $ for each $n \le i \le \ell-1$. 

Thus $\Ac \in \mathbf{SSS}$ with the desired exponents. 
\end{proof}

Recall the definitions of the parameters $n\le s$ in Theorem \ref{thm:CT-SSS}.
  \begin{theorem}
\label{thm:CS-IA}
Let $\I \subseteq \Phi^+(C_\ell)$ be an ideal. 
Then the toric arrangement  $\Ac_{\I }$ with respect to the root lattice is inductive with exponents
$$\exp(\Ac_{\I }) = \{0^{n-1} \} \cup \{ b_i\}_{i=n}^{s-1}  \cup \{ 2(\ell-i)\}_{i=s}^{\ell-1} \cup  \{\ell-s+1\}.$$ 
\end{theorem}

  \begin{proof}   Denote $\Ac :=\Ac_{\I }$.
  
  Case $1$.  First we prove the assertion when $s=1$. 
 In this case,  $\I = \Phi^+$. 
    We show that $\Ac  \in  \mathbf{IA}$ with the desired exponents by induction on $\ell$. 
    The case $\ell=1$ is clear. 
     
Suppose $\ell \ge 2$. 
Let $\delta: =2\epsilon_1=2\sum_{1\le k< \ell }\alpha_k+\alpha_\ell $ denote the highest root of $\Phi^+$. 
Define 
$$\scD:=\Phi^+ \setminus (E_1(\Phi^+)\cup \{\delta\}), \mbox{ and } \Dd : =  \Ac_{{\scD}}.$$
Then $\scD=\Phi^+(C_{\ell-1})$ (via $x_i \mapsto x_{i-1}$) . 
By the induction hypothesis, $\Dd  \in  \mathbf{IA}$ with  exponents
$$\exp(\Dd) = \{ 2(\ell-i)\}_{i=2}^{\ell-1} \cup  \{\ell-1\}.$$ 

Denote $\Ac' := \Ac  \setminus \{H_\delta\}$. 
Note that $\Ac' \setminus \Dd$ consists of the hypertori defined by the roots in $E_1(\Phi^+)$. 
These roots are given by
 \begin{align*}
\epsilon_1-\epsilon_j &=\sum_{1\le k<j }\alpha_k \quad (1  < j \le \ell),  \\
 \epsilon_1+\epsilon_j &=\sum_{1\le k<j }\alpha_k+2 \sum_{j\le k <\ell }\alpha_k+\alpha_\ell \quad(1   < j \le \ell) .
\end{align*}

Using a similar argument as in the proof of Lemma \ref{lem:CT-TM}, we may show that $\Dd$ is an M-ideal of $\Ac'$. 
Moreover, it is indeed a TM-ideal since Condition  \ref{def:M-TM-ideal-arr}(*) is satisfied because the coefficient at the simple $\alpha_1$ of all roots in $E_1(\Phi^+)$ is $1$, while that of the roots in $\scD$ is $0$.
Apply Lemma \ref{lem:Q<P} for $L(\Dd)$ and  $L(\Ac')$ we have that $\Ac' \in  \mathbf{IA}$ with exponents 
$$\exp(\Ac') = \exp(\Dd) \cup  \{2(\ell-1)\}  =  \{ 2(\ell-i)\}_{i=1}^{\ell-1} \cup  \{\ell-1\}.$$

Furthermore, one may check that the restriction $\Ac^{H_\delta}$ can be identified with $\Ac_{T_{\Phi^+}} (C_{\ell-1})$. 
(To see this just set $t_\ell = t_1^{-2}\cdots t_{\ell-1}^{-2}$ in the equations involving $t_\ell$. 
For example, the equation $t_2^2\cdots t_{\ell-1}^2t_\ell=1$ becomes $t_1^2=1$.) 
Thus by Theorem \ref{thm:CT-SSS}, $\Ac^{H_\delta} \in  \mathbf{IA}$ with exponents 
$$\exp(\Ac^{H_\delta}) =   \{ 2(\ell-i)\}_{i=1}^{\ell-1}.$$
   
Apply Theorem \ref{thm:IA-add}, we know that  $\Ac  \in  \mathbf{IA}$ with the desired exponents
$$\exp(\Ac) = \{ 2(\ell-i)\}_{i=1}^{\ell-1} \cup  \{\ell\}.$$

  Case $2$.  Now we prove the assertion when $s>1$. 
The set
$$\J : = \I \setminus \bigcup_{i=n}^{s-1} E_i(\I)$$ 
can be identified with $\Phi^+(C_{\ell-s+1})$. 
By   Case $1$ above, $\Pp: =\Ac_{{ \J}} \in  \mathbf{IA}$ with exponents 
$$
\exp(\Pp) = \{ 2(\ell-i)\}_{i=s}^{\ell-1} \cup  \{\ell-s+1\}.
$$

Using a similar argument as in Case $1$, we may show that the sets $E_i(\I)$ for $n \le i \le s-1$ give rise to a chain of TM-ideals for $\Ac$ starting from $\Pp$. 
Applying Lemma \ref{lem:Q<P} repeatedly, we may conclude that $\Ac  \in  \mathbf{IA}$ with the desired exponents. 
 \end{proof}

\begin{example}\label{eg:C-ideal}   
 Table \ref{tab:C5-ideal} shows an ideal $\I\subsetneq  \Phi^+(C_5)$ (in enclosed region) with $n=1$, $s=3$. 
By Theorem \ref{thm:CT-SSS}, $\Ac_{T_\I }\in \mathbf{SSS}$ with exponents $\{4,6,6,4,2\}$. 
By Theorem \ref{thm:CS-IA}, $\Ac_{\I }\in \mathbf{IA}$ with exponents $\{4,6,4,2,3\}$.
\end{example}

\begin{table}[htbp]
\centering
{\footnotesize\renewcommand\arraystretch{1.5} 
\begin{tabular}{ccccccc}
\mbox{Height} & & & && \\
9 &  $2\epsilon_1$ &   &  & && \\
8 & $\epsilon_1+\epsilon_2$ &   & & &&\\
7 &$\epsilon_1+\epsilon_3$ & $2\epsilon_2$  & && \\
\cline{3-3}
6 & $\epsilon_1+\epsilon_4$ & \multicolumn{1}{|c|}{$\epsilon_2+\epsilon_3$}  & && \\
 \cline{4-4}
5 &$\epsilon_1+\epsilon_5$ & \multicolumn{1}{|c}{$\epsilon_2+\epsilon_4$} &\multicolumn{1}{c|}{\cellcolor{red!25}{$2\epsilon_3$} } &&& \\
\cline{2-2}
4& \multicolumn{1}{|c}{$\epsilon_1-\epsilon_5$} & $\epsilon_2+\epsilon_5$& \multicolumn{1}{c|}{$\epsilon_3 +\epsilon_4$} &&&\\
\cline{5-5}
3 & \multicolumn{1}{|c}{$\epsilon_1-\epsilon_4$} & $\epsilon_2-\epsilon_5$ &$\epsilon_3 +\epsilon_5$ &  \multicolumn{1}{c|}{\cellcolor{red!25}{$2\epsilon_4$}} &&\\
2& \multicolumn{1}{|c}{$\epsilon_1-\epsilon_3$} & $\epsilon_2-\epsilon_4$ & $\epsilon_3-\epsilon_5$ &\multicolumn{1}{c|}{$\epsilon_4+\epsilon_5$}&&\\
\cline{6-6}
1 & \multicolumn{1}{|c}{$\epsilon_1-\epsilon_2$} & $\epsilon_2 -\epsilon_3$& $\epsilon_3-\epsilon_4$  &  $\epsilon_4-\epsilon_5$ & \multicolumn{1}{c|}{\cellcolor{red!25}{$2\epsilon_5$}}  &\\
\cline{2-6}
& $E_1$ & $E_2$ & $E_3$ & $E_4$  & $E_5 = \emptyset$ &  
\end{tabular}
}
\bigskip
\caption{An ideal $\I$ in $\Phi^+(C_5)$.}
\label{tab:C5-ideal}
\end{table}

\subsection{Type \texorpdfstring{$B$}{B}.}

The restriction of an ideal toric arrangement of type $B$ is in general not an ideal toric arrangement. 
We need an extension of the ideals so that the corresponding arrangements contain sufficient deletions and restrictions in order to apply the addition theorem \ref{thm:IA-add} to guarantee the inductiveness.

   \begin{lemma}
\label{lem:BT-IA}
Let $\I \subseteq \Phi^+(B_\ell)$ be an ideal such that $E^+_1(\I) \ne \emptyset $. 
Let $m=m(\I) $ be the integer so that  $\epsilon_1+\epsilon_m$ is the highest root in  $E^+_1(\I)$.
 (In particular, $2 \le m \le \ell$ and $2\ell-m =b_1$.)
 Let $1 \le p \le \ell+1$. Define the \emph{extension $ \I (p)$ of  $\I$ with parameter $p$} as follows:
  $$ \I (p) := \left( \I \setminus  \{ \epsilon_i  \mid p \le i \le \ell \} \right) \cup \{ 2\epsilon_i  \mid p \le i \le \ell \} .$$
If $m<p$, then  $\Ac_{T_{\I (p)}}$ is inductive with exponents 
$$\exp(\Ac_{T_{\I (p)}}) =\{2\ell-p+1\} \cup \{ b_i\}_{i=1}^{\ell-1}.$$ 
\end{lemma}

  \begin{proof}
  Denote $\Ac :=\Ac_{T_{\I (p)}}$. 
  We may write 
   $$\Ac = \Ac_{T_\I } \cup  \{ t_i =-1 \mid p \le i \le \ell \}.$$
   
   We show that $\Ac  \in  \mathbf{IA}$ with the desired exponents by induction on $\ell$. 
   If $\ell \le 2$, then $\Ac$ is always strictly supersolvable except when $p=3$ and $\I = \I(3) = \Phi^+(B_2)$. 
   In which case, $\Ac$ is indeed inductive with exponents $\{2,2\}$ by Figure \ref{fig:B2}.
   
   Now suppose $\ell\ge3$.
Since $\epsilon_1+\epsilon_m \in \I$, we must have $\epsilon_2+\epsilon_m \in \I$.
Define 
$$ \J : = \I \setminus (E_1(\I) \cup \{ \epsilon_1\}).$$ 
Then $\J$ can be regarded as an ideal of $\Phi^+(B_{\ell-1})$ (via $x_i \mapsto x_{i-1}$) with $m(\J) \le m(\I)-1$. 
Also, $E_i^\pm(\J) = E_{i+1}^\pm(\I) $ hence $b_i(\J) = b_{i+1}(\I) $ for all $1 \le i \le \ell-1$.

Moreover, $\I(p) \setminus (E_1(\I) \cup \{ \epsilon_1\})$ can be identified with the extension  $\J(p-1)$ since $2 \le m<p$. 
By the induction hypothesis, $\Pp: =\Ac_{T_{ \J(p-1)}} \in  \mathbf{IA}$ with exponents 
\begin{equation}
\label{eq:expP}
\exp(\Pp) = \{2\ell-p\} \cup \{ b_i(\I)\}_{i=2}^{\ell-1}.
\end{equation}

Define 
$$\scD:=\I(p) \setminus \{\epsilon_1+\epsilon_i   \mid m \le i \le p-1 \}, \mbox{ and } \Dd : =  \Ac_{T_{\scD}}.$$
Since $2\epsilon_i \in \scD$ for all $p \le i \le \ell$, using a similar argument as in the proof of Lemma \ref{lem:CT-TM} we may show that $\Pp$ is a TM-ideal of $\Dd$.
Apply Lemma \ref{lem:Q<P} for $L(\Dd)$ and  $L(\Pp)$ we have that $\Dd \in  \mathbf{IA}$ with exponents 
$$\exp(\Dd) = \exp(\Pp) \cup  \{2\ell-p+1\}  = \{2\ell-p+1,2\ell-p\} \cup \{  b_i(\I) \}_{i=2}^{\ell-1}.$$

Now we show that adding the $p-m$ hypertori $t_1t_{p-1}=1,t_1t_{p} =1 , \ldots, t_1t_m =1$ to $ \Dd$ in any order and applying Theorem \ref{thm:IA-add} to each addition step, we are able to conclude that  $\Ac  \in  \mathbf{IA}$ with the desired exponents. 
Since $2\ell-m =b_1$, it suffices to show that  the restriction at each addition step is inductive with exponents $ \{2\ell-p+1\} \cup \{  b_i(\I) \}_{i=2}^{\ell-1}.$

Indeed, the restriction at each step has the form $\Pp \cup \{H_k\}$ where $H_k$ denotes the hypertorus $t_k=-1$ for some $ m \le k \le p-1 $. 
Fix $ m \le k \le p-1 $. 
Note that $\epsilon_i+\epsilon_k \in \I \subseteq \J(p-1)$ for all $1 < i \ne k$ since $\epsilon_1+\epsilon_k \in \I$. 
Thus, the restriction $(\Pp \cup \{H_k\})^{H_k}$ can be identified with the arrangement $\Ac_{T_{ \scR(1)}}$, where $\scR(1)$ is the extension with parameter $p=1$ of an ideal $\scR$ of $\Phi^+(B_{\ell-2})$ (via $x_i \mapsto x_{i-1}$ ($2 \le i <k $) and $x_i \mapsto x_{i-2}$ ($k< i \le \ell$)) with
 $
b_i^\pm(\scR) = b_{i+1}^\pm(\I) -1
 $
 for $1 \le i \le \ell-2$. 
(Note that the equations  $
b_i^\pm(\scR) = b_{i+1}^\pm(\I) -1
 $
 for $k-1 \le i \le \ell-2$ follow from the fact that $\bigcup_{i=k-1}^{\ell-2} (E_i(\scR) \cup \{ 2\epsilon_i \})$ is a root system of type $C$.)
 
 Now using a similar argument as in the proof of Theorem \ref{thm:CT-SSS}, we know that $(\Pp \cup \{H_k\})^{H_k}$ is strictly supersolvable hence inductive with  exponents 
$$\exp((\Pp \cup \{H_k\})^{H_k})= \{  b_i(\scR)+2 \}_{i=1}^{\ell-2} =  \{  b_i(\I) \}_{i=2}^{\ell-1}.$$

By Theorem \ref{thm:IA-add} and Equation \eqref{eq:expP} above, we know that $\Pp \cup \{H_k\} \in \mathbf{IA}$ for every $ m \le k \le p-1 $ with the desired exponents 
$$\exp(\Pp \cup \{H_k\})=  \{2\ell-p+1\} \cup \{  b_i(\I) \}_{i=2}^{\ell-1}.$$

This completes the proof.
\end{proof}

   \begin{theorem}
\label{thm:BT-IA}
Let $\I \subseteq \Phi^+(B_\ell)$ be an ideal such that $\epsilon_k \in \I$ for some $1 \le k \le \ell$.
Define 
\begin{align*}
n & :=  \min \{ 1 \le i \le \ell \mid E_i(\I) \ne \emptyset \}, \\
a & :=  \min \{ n \le i \le \ell \mid \epsilon_i \in \I \mbox{ and } E_i^+(\I) = \emptyset \},\\
s & :=  \min \{ a \le i \le \ell \mid E_i^+(\I) \ne \emptyset \}.
\end{align*}
For each $s \le i \le \ell$, let $m(i) $ be the integer so that  $\epsilon_i+\epsilon_{m(i)}$ is the highest root in  $E_i(\I)$. 
 (In particular, $m(j) \le m(i)$ if $i<j$.)
Let $s \le p \le \ell+1$, recall the definition of the extension $ \I (p)$ of  $\I$ with parameter $p$ in Lemma \ref{lem:BT-IA}. Define 
 \begin{align*}
t  :=  \min \{ s \le i \le \ell \mid m(i) <p \}.
\end{align*}
Then  $\Ac_{T_{\I (p)}}$ is inductive with exponents 
$$\exp(\Ac_{T_{\I (p)}}) = \{0^{n-1}, 2\ell-p-t+2\} \cup \{ b_i+1 \mid i \in [a, t-1]\} \cup \{ b_i \mid i \in [n, \ell-1] \setminus [a, t-1]\ \}.$$
\end{theorem}

  \begin{proof}
    Denote $\Ac :=\Ac_{T_{\I (p)}}$. 
The set
$$(\I(p) \setminus \bigcup_{i=n}^{a-1} E_i(\I)) \setminus \bigcup_{i=a}^{t-1}  (E_i(\I)) \cup \{ \epsilon_i\})$$ 
can be identified with the extension $\J(p-t+1)$, where $\J$ is an ideal of $\Phi^+(B_{\ell-t+1})$  with $m(i) < p-t+1$ for all $1 \le i \le \ell-t+1$. 
By Lemma \ref{lem:BT-IA}, $\Pp: =\Ac_{T_{ \J(p-t+1)}} \in  \mathbf{IA}$ with exponents 
$$
\exp(\Pp) = \{2\ell-p-t+2\} \cup \{ b_i(\I)\}_{i=t}^{\ell-1}.
$$

Using a similar argument as in the proof of Lemma \ref{lem:CT-TM}, we may show that the sets $E_i(\I)$ for $n \le i \le a-1$ and $E_i(\I) \cup\{\epsilon_i\}$ for $a \le i \le t-1$ give rise a chain of TM-ideals for $\Ac$ starting from $\Pp$. (Note that by definition $m(i) \ge p$ for all $s \le i \le t-1$.) 
Applying Lemma \ref{lem:Q<P} repeatedly we may conclude that $\Ac  \in  \mathbf{IA}$ with the desired exponents. 
Indeed, the sets above contribute to $\exp(\Ac)$ the exponents $b_i$ for $n \le i \le a-1$ and $b_i+1$ for $a \le i \le t-1$. 
\end{proof}

\begin{example}\label{eg:B-ideal}   
 Table \ref{tab:B5-ideal} shows the extension $\I(4)$ of an ideal $\I\subsetneq  \Phi^+(B_5)$ with parameter $p=4$. 
 In this case, $n=a=s=1$ and $t = 2$ with $m(t)=3<p$. 
 By Theorem \ref{thm:BT-IA}, 
 $\Ac_{T_{\I(4)} }\in \mathbf{IA}$ with exponents $\{6,7,6,4,2\}$.
\end{example}

\begin{table}[htbp]
\centering
{\footnotesize\renewcommand\arraystretch{1.5} 
\begin{tabular}{ccccccc}
\mbox{Height} & & & && \\
9 &   $\epsilon_1+\epsilon_2$ &   &  & && \\
8 &  $\epsilon_1+\epsilon_3$ &   & & &&\\
\hhline{~--}
7 & \multicolumn{1}{|c}{\cellcolor{blue!25}{$\epsilon_1+\epsilon_4$}} & \multicolumn{1}{c|}{\cellcolor{blue!25}{$\epsilon_2+\epsilon_3$}}  & && \\
6 & \multicolumn{1}{|c}{\cellcolor{blue!25}{$\epsilon_1+\epsilon_5$}} & \multicolumn{1}{c|}{\cellcolor{blue!25}{$\epsilon_2+\epsilon_4$}}  & && \\
\cline{4-4}
5 & \multicolumn{1}{|c}{\cellcolor{red!25}{$\epsilon_1$}} & \cellcolor{blue!25}{$\epsilon_2+\epsilon_5$} &\multicolumn{1}{c|}{\cellcolor{blue!25}{$\epsilon_3 +\epsilon_4$}} &&& \\
4& \multicolumn{1}{|c}{\cellcolor{yellow!25}$\epsilon_1-\epsilon_5$} & \cellcolor{red!25}{$\epsilon_2$} & \multicolumn{1}{c|}{\cellcolor{blue!25}{$\epsilon_3 +\epsilon_5$}} &&&\\
\cline{5-5}
3 & \multicolumn{1}{|c}{\cellcolor{yellow!25}$\epsilon_1-\epsilon_4$} & \cellcolor{yellow!25}$\epsilon_2-\epsilon_5$ & \cellcolor{red!25}{$\epsilon_3$} &  \multicolumn{1}{c|}{\cellcolor{blue!25}{$\epsilon_4+\epsilon_5$}} &&\\
2& \multicolumn{1}{|c}{\cellcolor{yellow!25}$\epsilon_1-\epsilon_3$} & \cellcolor{yellow!25}$\epsilon_2-\epsilon_4$ & \cellcolor{yellow!25}$\epsilon_3-\epsilon_5$ & \multicolumn{1}{c|}{\cellcolor{red!25}{$2\epsilon_4$}} &&\\
\cline{6-6}
1 & \multicolumn{1}{|c}{\cellcolor{yellow!25}$\epsilon_1-\epsilon_2$} & \cellcolor{yellow!25}$\epsilon_2 -\epsilon_3$& \cellcolor{yellow!25}$\epsilon_3-\epsilon_4$  & \cellcolor{yellow!25}$\epsilon_4-\epsilon_5$ & \multicolumn{1}{c|}{\cellcolor{red!25}{$2\epsilon_5$}} &\\
\cline{2-6}
\end{tabular}
}
\bigskip
\caption{Extension of an ideal $\I$ in $\Phi^+(B_5)$ with parameter $p=4$.}
\label{tab:B5-ideal}
\end{table}

 Recall from Remark \ref{rem:A-const} that $\Ac_{{\Psi}  }$ and  $\Ac_{T_\Psi }$ have isomorphic poset of layers  for every $\Psi\subseteq\Phi^+(B_\ell)$.
 
   \begin{corollary}
\label{cor:BST-IA}
If $\I \subseteq \Phi^+(B_\ell)$, then the toric arrangement $\Ac_{{\I}  }$  with respect to the root  lattice is inductive. 
\end{corollary}
 
  \begin{proof}
    If $\epsilon_i\notin \I$ for all $1 \le i \le \ell $, then $\I$ can be regarded as an ideal of $\Phi^+(A_{\ell-1})$. 
    Thus $\Ac_{\I }$ is indeed strictly supersolvable hence inductive by Corollary \ref{cor:ideal-SS-A}.  
Otherwise, we know that $\Ac_{T_\I }$ is inductive which follows from Theorem \ref{thm:BT-IA} by letting $p=\ell+1$.
\end{proof}

\begin{example}\label{eg:BC-root}   
From Theorems \ref{thm:CS-IA}, \ref{thm:BT-IA} and Corollary \ref{cor:BST-IA}, we deduce that both $\Ac_{{ \Phi^+}  }(B_\ell)$ and  $\Ac_{ \Phi^+}(C_\ell)$ are inductive with the same multiset of exponents $ \{ \ell, 2, 4, \ldots, 2(\ell-1)\}.$ 
This fact is similar to the hyperplane arrangement case.

\end{example}

In contrast to the inductiveness, the toric arrangement of a root system of type $B_\ell$ is not supersolvable for most cases.

\begin{theorem}
\label{cor:BST-NSS}
Suppose $\Phi =B_\ell$ for $\ell \ge 1$. 
Then $\Ac_{T_{\Phi^+}}$ is supersolvable if and only if $\ell \le 2$. 
\end{theorem}
\begin{proof}
Let $\Ac :=\Ac_{T_{\Phi^+}}$. 
Denote $L = L(\Ac)$ and $x = (-1,-1,\ldots,-1) \in L$. 
By Lemma \ref{lem:loc-res}, $L_{\leq x}$ is isomorphic to the intersection lattice $L(\Hc_{T_{\Phi^+}}(D_\ell))$ of the hyperplane arrangement of a root system of type $D_\ell$. 

If $\ell \ge 4$, then $L_{\leq x}$ is not supersolvable by Remark \ref{rem:all-property}. 
Therefore, $L$ is not locally supersolvable hence not supersolvable. 

When $\ell \le 3$,  however, $L_{\leq x}$ is always supersolvable. We need a direct examination for the supersolvability of $L$.
The assertion is clear when $\ell=1$. 
The case $\ell=2$ is shown in Figure \ref{fig:B2}.

Now we show that $L$ is not supersolvable (though locally supersolvable) when $\ell=3$ by showing that $L$ does not have an M-ideal of rank $2$. 

Suppose to the contrary that such an M-ideal exists and call it $\Qc$. 
Denote  $H^+_{ij} := \{ t_it_j =1\} $ and  $H^-_{ij} :=\{t_it_j^{-1} = 1\}$.
First, notice that a rank-$2$ element of the form $t_i = t_j = -1$ covers exactly two atoms,  namely $H^+_{ij}$ and $H^-_{ij}$. 
If these atoms are not in $\Qc$, then Lemma \ref{lem:GPMI} fails. 
Hence, at least one of them belongs to $\Qc$ for every pair of indices $i \neq j \in \{1,2,3\}$. 
Moreover, we may deduce that exactly one  of $H^+_{ij}$ and $H^-_{ij}$ belongs to $\Qc$. 
Otherwise, the join $H^+_{ij} \vee H^-_{ij} \vee H$ where $H$ is either $H^+_{jk}$ or $H^-_{jk}$ for $k \notin \{i,j\}$ contains an element of rank $3$, which contradicts the join-closedness of $\Qc$. 

We consider two main cases, the remaining cases are similar to one of these.
\begin{enumerate}[(a)]
	\item If $H^+_{12}$, $H^+_{13}$, $H^+_{23}$ all belong to $\Qc$, then their join consists of rank-$3$ elements, a contradiction.
	\item If $H^+_{12}$, $H^+_{13}$, $H^-_{23}$ all belong to $\Qc$, then $\Qc$ has no atom of the form $t_i = 1$, otherwise joining it with $H^+_{12} \vee H^+_{13} \vee H^-_{23}$ would give a rank-$3$ element inside $\Qc$. 
	Hence, the only rank-$2$ element in $\Qc$ would be $H^+_{12} \vee H^+_{13} \vee H^-_{23} = \{t_2 = t_3 = t_1^{-1} \}$. 
	However, this is not an element of $L_{\leq (1,-1,-1)}$, which contradicts Condition \ref{def:M-TM-ideal}(2).
\end{enumerate}

This completes the proof.
\end{proof}

 \vskip 1em
\noindent
\textbf{Acknowledgements.} 
The authors would like to thank Shuhei Tsujie for many stimulating discussions. 
The third author was supported by a postdoctoral fellowship of the Alexander von Humboldt Foundation at Ruhr-Universit\"at Bochum.
The major part of this paper was done during the visit of the third author at the Bologna University supported by the TInGeCoRe internationalization project. 
He greatly thanks the Bologna University for the hospitality. 

 \bibliographystyle{abbrv}
\bibliography{references}

\end{document}